\documentclass{imsart}

\DeclareUnicodeCharacter{00A0}{~}

\RequirePackage{amsthm,amsmath,amsfonts,amssymb}
\RequirePackage[numbers,sort&compress]{natbib}
\RequirePackage[colorlinks,citecolor=blue,urlcolor=blue]{hyperref}
\RequirePackage{graphicx}
\usepackage{enumerate}
\startlocaldefs
\theoremstyle{plain}

\DeclareMathOperator{\var}{\mathbb Var}

\DeclareMathOperator{\tr}{tr}

\DeclareMathOperator{\sign}{sign}

\DeclareMathOperator{\perm}{perm}
\newcommand{\T}{\top}

\newcommand{\HH}{\mathbb H}
\newcommand{\HHd}{\mathbb H_\delta}
\newcommand{\PP}{\mathbb P}

\newcommand{\cA}{{\mathcal A}}
\newcommand{\cB}{{\mathcal B}}

\newcommand{\cD}{{\mathcal D}}
\newcommand{\cE}{{\mathcal E}}

\newcommand{\cI}{{\mathcal I}}

\newcommand{\cK}{{\mathcal K}}

\newcommand{\cO}{{\mathcal O}}

\newcommand{\bs}{\boldsymbol}
\newcommand{\one}{\bs 1}
\newcommand{\bA}{\bs A}

\newcommand{\bI}{\bs I}
\newcommand{\bJ}{\bs J}

\newcommand{\bP}{\bs P}

\newcommand{\bS}{\bs S}
\newcommand{\uS}{\mathsf S}

\newcommand{\ba}{\bs a}
\newcommand{\bd}{\bs d}

\newcommand{\bm}{\bs m}
\newcommand{\bp}{\bs p}
\newcommand{\bq}{\bs q}
\newcommand{\bss}{\bs s}
\newcommand{\obJ}{\overset{\circ}{\bJ}}

\newcommand{\C}{\mathbb{C}}
\newcommand{\N}{\mathbb{N}} 
\newcommand{\E}{\mathbb{E}}

\newtheorem{thm}{Theorem}[section]
\newtheorem{lem}[thm]{Lemma}

\newtheorem{prop}[thm]{Proposition}
\newtheorem{cor}[thm]{Corollary}

\theoremstyle{definition}

\theoremstyle{remark}
\newtheorem{rem}[thm]{Remark}

\newtheorem{assumption}[thm]{Assumption}

\newcommand{\toprobalong}{\xrightarrow[n\to\infty]{{\mathcal P}}}  
\newcommand{\toprobashort}{\xrightarrow[]{{\mathcal P}}}  
\newcommand{\R}{\mathbb{R}}

\newcommand{\vertiii}[1]{{\left\vert\kern-0.25ex\left\vert\kern-0.25ex\left\vert #1
    \right\vert\kern-0.25ex\right\vert\kern-0.25ex\right\vert}}
    \newcommand{\1}{\mathbbm{1}}
    \newcommand{\mathbbm}[1]{\text{\usefont{U}{bbm}{m}{n}#1}} 
\theoremstyle{remark}

\endlocaldefs

\begin{document}

\begin{frontmatter}
\title{On the spectral radius and the characteristic polynomial 
of a random matrix with independent elements \\ 
and a variance profile}
\runtitle{Spectral radius of a random matrix with a variance profile}

\begin{aug}

\author{\fnms{Walid}~\snm{Hachem} }\ead[label=e1]
{walid.hachem@univ-eiffel.fr},
 \author{\fnms{Michail}~\snm{Louvaris}}\ead[label=e2]{louvarismixalis@gmail.com}
\address[]{CNRS, Laboratoire d’informatique Gaspard Monge (LIGM / UMR 8049),
 Universit\'e Gustave Eiffel, ESIEE Paris, France.
 \printead[presep={ ,\ }]{e1,e2}}

\end{aug}

\begin{abstract}
In this paper, it is shown that with large probability, the spectral radius of
a large non-Hermitian random matrix with a general variance profile does not
exceed the square root of the spectral radius of the variance profile matrix. A
minimal moment assumption is considered and sparse variance profiles are
covered.  Following an approach developed recently by Bordenave, Chafaï and
Garc\'{\i}a-Zelada, the key theorem states the asymptotic equivalence between
the reverse characteristic polynomial of the random matrix at hand and a random
analytic function which depends on the variance profile matrix. The result is
applied to the case of a non-Hermitian random matrix with a variance profile
given by a piecewise constant or a continuous non-negative function, the
inhomogeneous (centered) directed Erdős–Rényi model, and more. 
\end{abstract}



\end{frontmatter}

\section{Problem description and results} 
\label{sec-pb} 

Let $( W_{ij} )_{ij\geq 1}$ be an infinite array of complex-valued independent
and identically distributed random variables such that $\E W_{11} = 0$ and $\E
|W_{11}|^2 = 1$.  For each integer $n > 0$, let $S^{(n)} = \bigl[ s^{(n)}_{ij}
\bigr]_{i,j=1}^n$ be a $n\times n$ deterministic matrix with real non-negative
elements.  Consider the $\mathbb{C}^{n\times n}$--valued random matrix $X^{(n)}
= \begin{bmatrix} X^{(n)}_{ij} \end{bmatrix}_{i,j=1}^n$ which elements are
defined as 
\[
X^{(n)}_{ij} = \sqrt{s^{(n)}_{ij}} W_{ij} . 
\]
The purpose of this paper is to study the large--$n$ behavior of the spectral 
radius $\rho(X^{(n)})$ under general assumptions on the sequence of
matrices $(S^{(n)})$ that cover the sparse cases. These assumptions stand as 
follows: 
\begin{assumption} 
\label{bnd-S} 
The following hold true. 
\begin{enumerate}[(i)] 
\item\label{rowsum}
There exists a constant $C_S > 0$ such that at least one of the following 
bound holds: 
\[
\vertiii{S^{(n)}} \leq C_S \quad \text{or} \quad 
\vertiii{(S^{(n)})^\top} \leq C_S, 
\]
where $\vertiii{\cdot}$ is the max row $\ell^1$--norm of a matrix.
\item\label{bnd-sij}
There exists a positive sequence $(K_n)_{n\geq 1}$ tending to infinity, 
and there exists a constant $C'_S > 0$ such that
\[
s_{ij}^{(n)} \leq \frac{C'_S}{K_n}
\]
for all $n$ and all $i,j \in [n]$.
\end{enumerate} 
\end{assumption} 
\begin{assumption}
\label{bnd-det}
For each $\varepsilon \in(0,1]$, it holds that 
\[
 \liminf_n \min_{\gamma\in[0,1-\varepsilon]} 
       \det\left( I_n - \gamma S^{(n)} \right)  > 0 .
\]
\end{assumption}
Intuitively, this assumption is related with a rapid decrease of the magnitudes of the eigenvalues of $S^{(n)}$, as illustrated by the examples below.

This assumption can be re-expressed as 
\begin{subequations} 
\label{other} 
\begin{align}
& \limsup_n \rho(S^{(n)}) \leq 1, \quad \text{and} \label{rho<1} \\ 
& \forall\varepsilon \in (0,1], \ \liminf_n 
       \det\left( I_n - (1-\varepsilon) S^{(n)} \right)  > 0.  
\label{2bnd-det} 
\end{align}
\end{subequations}

Indeed, when Assumption~\ref{bnd-det} holds true, \eqref{2bnd-det} is obvious.
Moreover, since $S^{(n)}$ has non-negative elements, it has an eigenvalue which
is equal to $\rho(S^{(n)})$ \cite[Th.~8.3.1]{hor-joh-livre90},
hence~\eqref{rho<1}. Conversely, assume the conditions~\eqref{other} are 
satisfied. Since $\limsup \rho(S^{(n)}) \leq 1$, it holds that for each 
$\gamma\in[0,1)$, the series $\sum_{k\geq 1} \gamma^k \tr (S^{(n)})^k / k$ is 
convergent for each large enough $n$, and we can write for these $n$: 
\begin{equation}
\label{dev-det} 
\forall \gamma\in[0,1), \quad 
\det(I_n - \gamma S^{(n)}) = \exp\left( - \sum_{k=1}^\infty 
  \gamma^k \frac{\tr (S^{(n)})^k}{k} \right). 
\end{equation} 
This shows that $\gamma\mapsto \det(I_n - \gamma S^{(n)})$ is a non-negative
decreasing function on $[0,1)$, and~\eqref{2bnd-det} implies 
Assumption~\ref{bnd-det}. 

The following theorem is established in this paper: 
\begin{thm}
\label{th-main}
Let Assumptions~\ref{bnd-S} and \ref{bnd-det} hold true. Then, it holds that
\[
\forall \varepsilon > 0, \quad 
\PP\left[\rho\left(X^{(n)}\right) \geq 1 + \varepsilon \right] 
 \xrightarrow[n\to\infty]{} 0 . 
\]
\end{thm}

Let us sketch some application examples to shed some light on the assumptions
and the result. These will be clarified and detailed in the next section. To
begin with, assume that $S^{(n)}$ is a block variance profile matrix with a 
fixed number of rectangular blocks which dimensions are of order $n$, and which
elements are of order $1/n$.  Then, $S^{(n)}$ satisfies Assumption~\ref{bnd-S}
with $K_n = n$.  Assume that the spectral radius of $S^{(n)}$ is of order one.
By normalizing this matrix with its spectral radius, Condition~\eqref{rho<1} is
satisfied.  Furthermore, since the rank of $S^{(n)}$ is bounded by a constant,
Condition~\eqref{2bnd-det} is also satisfied, and Theorem~\ref{th-main} asserts
that with high probability, $\rho(X^{(n)})$ cannot be larger and away of
$\sqrt{\rho(S^{(n)})}$ before the normalization. 

A similar conclusion can be obtained when $S^{(n)}$ is obtained by a regular
sampling of a continuous non-negative function on the rectangle $[0,1]^2$.
Assumption~\ref{bnd-S} will still be satisfied with $K_n = n$.  Let us turn to
Condition~\eqref{2bnd-det}. Here, the rank of $S^{(n)}$ is no more necessarily
bounded.  However, this condition will still be satisfied after the proper
normalization because, roughly speaking, $S^{(n)}$ will have only a few
non-negligible eigenvalues. This is due to the fact that $S^{(n)}$ is a
discrete approximation of a continuous function on $[0,1]^2$, which is as is
well known a compact trace-class operator on the Banach space $C([0,1])$ of 
the continuous functions on $[0,1]$. 

Of particular interest are the situations where $K_n = o(n)$, that we refer to
as the ``sparse'' cases, where, typically, the number of non zero elements of
$S^{(n)}$ per row belongs to the interval $[cK_n, C K_n]$ where $0 < c < C <
\infty$, and these elements are bounded by $C'/K_n$ for $C' > 0$. Examples of
these cases where Assumption~\ref{bnd-det} is satisfied will be detailed below. 
\\

Theorem~\ref{th-main} only provides an upper bound on the spectral radius
$\rho(X^{(n)})$.  One can expect the stronger result
$\PP\left[|\rho\left(X\right) - 1 | \geq \varepsilon \right] \to 0$, which
requires showing that $$\PP\left[\rho\left(X^{(n)}\right) \leq 1 - \varepsilon
\right] \to 0$$ when $\rho(S^{(n)})$ is close to one.  One way of establishing
this last result is to establish a so-called global law on the spectral measure
of $X^{(n)}$, showing that this spectral measure can be approximated for all
large $n$ with a distribution supported by the closed unit disk.  For matrices
with variance profiles, global laws were established in the literature in some
non-sparse situations where $K_n = n$ and where the $\ell^1$ norms of all the
rows and the columns of $S^{(n)}$ are of order one. The first of such results
was revealed by Girko \cite{gir-livre01}. A global law was rigorously
established by Alt \emph{et.al.}~in \cite{alt-erd-kru-18} under moment and
density assumptions and in the case where all the numbers $n s^{(n)}_{ij}$
belong to a compact interval lying away from zero (the so-called ``flat''
variance profile).We note that, beyond the global law, the large-$n$ behavior of the spectral radius is also analyzed in~\cite{alt-erd-kru-18} and \cite{cipolloni2312universality} through the establishment of local law results. For a slightly more general model, see also Remark 1.7 in~\cite{bao2025signal}. Regarding the spectral radius in \cite{alt-erd-kru-21} the authors prove a result analogous to Theorem \ref{th-main} under the assumption of a "flat" variance profile. Our result extends [\cite{alt-erd-kru-21}, Theorem 2.1] beyond the flat variance profile case, as it encompasses a broader class of models, see Examples \ref{erdos_reyni_subsection} and \ref{d-reg-out}, and is derived using more elementary methods. Beyond the flat variance profile in the non-sparse
case, the global law was established in \cite{coo-hac-naj-ren-18} under a
so-called robust irreducibility assumption on the variance profile, and a
moment assumption on the matrix entries.   \\ 

Regarding the applications, the control of the spectral radius of $X^{(n)}$ is
essential in the study of many dynamical systems that arise in the fields of
control theory, natural or artificial neural networks, theoretical biology and
ecology, and others.  For instance, in neural networks, $X^{(n)}$ is the matrix
that represents the couplings between $n$ neurons \cite{som-cri-som-88}.  In
theoretical ecology, $X^{(n)}$ is used to model the random food interactions
between $n$ living species that coexist within an ecosystem \cite{akj-etal-24}.
In these situations, the inhomogeneity of the matrix model represented by the
variance profile, or the sparsity of its non-zero elements are often advocated
to model realistic situations.  The transition of such systems from
stationary to chaotic dynamics is often driven by $\rho(X^{(n)})$. \\ 

To obtain Theorem~\ref{th-main}, we use the approach based on the
reverse characteristic polynomial of $X^{(n)}$ that Bordenave, Chafaï and
Garc\'{\i}a-Zelada developed in \cite{bordenave2021convergence} to deal with
the case $s_{ij}^{(n)} = 1/n$, and that was partially inspired by the article
\cite{bas-zei-20} devoted to a different problem. Observe that no moment of the
random variables $W_{ij}$ beyond the second moment is required in the statement
of Theorem~\ref{th-main}.  This is a prominent feature of the approach
of~\cite{bordenave2021convergence}, which improves upon the older literature
such as \cite{bai-yin-86,gem-86,bor-cap-cha-tik-18}.  In the recent literature,
the approach of~\cite{bordenave2021convergence} for controlling the spectral
radius was applied for the Elliptic Ginibre model in \cite{fra-gar-23}.  In the
same vein, the recent papers \cite{coste2023sparse} and \cite{cos-lam-yiz-24},
consider the characteristic polynomial of sparse Bernoulli matrices and sums of
random permutations and regular digraphs.

In the remainder, we fix a real number $\delta > 0$ to a value as small as we
wish, and we denote as $\HHd$ the space of holomorphic functions on the open
disk $D(0,1-\delta)$ of $\C$ with center zero and radius $1-\delta$ equipped
with the topology of the uniform convergence on the compacts of
$D(0,1-\delta)$. We also denote as $\HH$ the space of holomorphic functions on
the open unit-disk $D(0,1)$.  As is well known, these spaces are Polish spaces.

Let us consider the reverse characteristic polynomial of the matrix $X^{(n)}$, 
which is defined as 
\[
q_n(z)= \det\left( I_n -z X^{(n)}\right) .
\]
Obviously, $q_n$ is a $\HHd$--valued random variable.  Our paper is mainly
devoted towards studying the asymptotic behavior of the probability
distribution of $q_n$ on $\HHd$.  Here, a notation is in order.  Let $(U_n)$ and
$(V_n)$ be two sequences of random variables valued in some metric space. For
each $n$, let $\mu_n$ and $\nu_n$ be the probability distributions of $U_n$ and
$V_n$ respectively. We shall use the notation 
\[
  U_n \sim_n V_n
\]
to refer to the facts that the sequences $(\mu_n)$ and $(\nu_n)$ are relatively
compact, and that 
\[
   \int f d\mu_n - \int f d \nu_n \xrightarrow[n\to\infty]{}  0
\]
for each bounded continuous real function $f$ on the metric space. We shall say
then that $(U_n)$ and $(V_n)$ are ``asymptotically equivalent''. Note that
$(\mu_n)$ and $(\nu_n)$ do not necessarily converge narrowly to some
probability distribution. This setting is well-suited to describe the
asymptotics of our sequence $(q_n)$ because without an additional assumption on
the construction of the sequence $(S^{(n)})$, the distribution of $q_n$ has no
reason to converge narrowly to a limit probability measure on $\HHd$.
These asymptotics are described by the following theorem, which will be proven
in Section~\ref{sec-prfmain}: 
\begin{thm}
\label{th-chpol}
Let Assumptions~\ref{bnd-S} and~\ref{bnd-det} hold true. Then, for all large
$n$, the function 
\[
 \kappa_n(z)= \sqrt{\det(I- z^2\E W^2_{1,1}S^{(n)})} 
\] 
is a well-defined element of $\HHd$ with the square root being the one for 
which $\kappa_n(0) = 1$. The function 
\[
F_n(z)=\sum_{k=1}^{\infty} z^k Z_k \sqrt{\frac{\tr((S^{(n)})^k)}{k}}, 
\] 
where $(Z_k)_{k = 1,2,\ldots}$ is a sequence of independent complex Gaussian 
random variables such that
\[
\E Z_k=0, \quad \E |Z_k|^2=1, \quad \text{and} \quad 
  \E Z_k^2= (\E W^2_{11} )^{k}. 
\] 
is a well-defined $\HHd$--valued random variable. The sequence $(F_n)$ is tight 
in $\HHd$, and the sequence $(\kappa_n)$ satisfies for each compact set 
$\cK \subset D(0,1-\delta)$:
\begin{equation}
\label{bnd-kappa}
0 < \liminf_n \min_{z\in\cK} |\kappa_n(z)| \leq
    \limsup_n \max_{z\in\cK} |\kappa_n(z)| < \infty.
\end{equation} 
Finally, it is true that 
\begin{equation} 
\label{asymptotic_equivalence_of_qn}
   q_n(z) \sim_n \kappa_n(z) \exp (-F_n(z)) 
\end{equation} 
as $\HHd$--valued random variables. 
\end{thm}

Theorem~\ref{th-main} can be deduced from Theorem~\ref{th-chpol} by an argument
provided in Section~\ref{subs_proof_of_spectral_radius}.  

\begin{rem}
\label{H-Hd} 
In the case where the variances of the elements of $X^{(n)}$ are equal to 
$1/n$, we have $S^{(n)} = n^{-1} \one_n \one_n^\T$ where $\one_n$ is the
vector of all ones in $\R^n$. In this case, it is obvious that $\kappa_n$ and
$F_n$ are independent of $n$ and are given as 
$\kappa_n(z) = \kappa(z) := \sqrt{1 - z^2 \E W^2_{1,1}}$ and 
$F_n(z) = F(z) := \sum_{k=1}^\infty z^k Z_k / \sqrt{k}$. By a straightforward 
modification of the proof of Theorem~\ref{th-chpol}, the asymptotic
equivalence~\eqref{asymptotic_equivalence_of_qn} in the space of 
$\HHd$--valued random variables can be replaced with the convergence in 
distribution of $q_n$ towards $\kappa \exp(-F)$ in the space of $\HH$--valued 
random variables, as stated in~\cite{bordenave2021convergence}. 
In our more general situation, we need to replace $\HH$ with $\HHd$, as the 
following example shows: when $\E W_{11}^2 = 1$ and $\rho(S^{(n)}) > 1$ for 
each $n$, no function $\kappa_n \exp(-F_n)$ belongs to $\HH$. We thank one 
of the referees for pointing out this remark. 
\end{rem}

\section{Case studies} 
\label{sec-app} 

In this section, we describe some matrix models for which
Assumptions~\ref{bnd-S} and~\ref{bnd-det} are satisfied.  The proofs related
with this section are provided in Section~\ref{prf-cases}.

\subsection{Block variance profile}
Fix $d, m \in \N\setminus\{0\}$, and let $\bA$ be a $d \times m$ matrix with 
positive entries. Set $p = md$. For each $n > 0$, define the matrix $A^{(pn)}$ 
as 
\begin{align*}
A^{(pn)} = \frac{1}{pn} \bA \otimes (\one_{mn} \one_{dn}^\T) \in 
 \R^{pn \times pn},  
\end{align*}
where $\otimes$ denotes the Kronecker product of matrices. Thus, $A^{(pn)}$ 
consists in $md$ rectangular blocks of positive numbers, and the dimensions of 
each block scale with $n$. 

\begin{prop}\label{Pw-Cst}
In the above setting, $\rho(A^{(pn)}) > 0$, and the matrix 
$$S^{(pn)} = \rho(A^{(pn)})^{-1} A^{(pn)}$$ satisfies Assumptions~\ref{bnd-S} 
and~\ref{bnd-det}. 
\end{prop}
\begin{proof}
We check these assumptions with $K_n = n$.  Since all the elements of $\bA$ are
positive, the minimum row sum of $A^{(pn)}$ lies in a compact interval of
$\R_+$ away from zero. Thus, $\rho(A^{(pn)}) \geq c$ for some constant $c > 0$
\cite[Th.~8.1.22]{hor-joh-livre90}.  As a consequence, $S^{(pn)}$ exists and 
complies with Assumptions~\ref{bnd-S}. Furthermore, $\rho(S^{(pn)}) = 1$, and 
the rank $r$ of $S^{(pn)}$ is upper bounded by $\min(d,m)$. 
Assumption~\ref{bnd-det} follows from the inequality
$\det(I-\gamma S^{(pn)}) \geq (1-\gamma)^{r}$ for $\gamma\in[0,1)$. 
\end{proof}
We thus obtain from Theorem~\ref{th-main} that 
$\PP[\rho(X^{(pn)}) \geq 1 + \varepsilon] \to_n 0$ for each $\varepsilon > 0$. 
If we take out the normalization by $\rho(A^{(pn)})$ in the construction of 
$S^{(pn)}$, we of course obtain that 
$\PP[\rho(X^{(pn)}) \geq \sqrt{\rho(A^{(pn)})} + \varepsilon] \to_n 0$ for 
each $\varepsilon > 0$.

\subsection{Sampling a continuous variance profile}
\label{S-cont}

It is well known that any continuous function $\bS : [0,1]^2 \to \R_+$, seen as
an integral operator on the Banach space $C([0,1])$, is a compact trace-class
operator \cite{sim-livre05}. If its spectral radius $\rho(\bS)$ is positive, we
normalize our operator with $\rho(\bS)$ which amounts to assuming that
$\rho(\bS) = 1$. If $\rho(\bS) = 0$, then, we replace $\bS$ with $C_{\bS} \bS$
where $C_{\bS} > 0$ is an arbitrarily large constant. 

For $n > 0$, our variance profile matrix $S^{(n)}$ will be obtained by sampling
regularly the function $\bS$ on the rectangle $[0,1]^2$, namely, by setting
$s^{(n)}_{ij} = n^{-1} \bS(i/n, j/n)$. 

It is obvious that Assumption~\ref{bnd-S} is satisfied by $S^{(n)}$ with $K_n =
n$. The validity of Assumption~\ref{bnd-det} is the object of the following
proposition which proof follows from standard arguments. We include it in
Section~\ref{prf-cases} for completeness. 
\begin{prop} 
\label{TC} 
In the setting described above, it holds that 
\[
\lim_n \min_{\gamma\in[0,1-\varepsilon]} \det(I_n - \gamma S^{(n)}) 
= \lim_n \det(I_n - (1-\varepsilon) S^{(n)}) 
= \det(\bI - (1-\varepsilon) \bS) > 0 
\]
for each $\varepsilon\in (0,1]$, where $\bI$ is the identity operator on the 
Banach space $C([0, 1])$, and $\det\left( \bI - (1-\varepsilon) \bS \right)$ is 
a Fredholm determinant. 
\end{prop} 
Thus, Assumption~\ref{bnd-det} holds true, and Theorem~\ref{th-main} follows.
Consequently, if we get back to our original operator $\bS$ (before the
multiplication by $\rho(\bS)^{-1}$ or by $C_{\bS}$), we obtain that
$\PP[\rho(X^{(n)}) \geq \sqrt{\rho(\bS)} + \varepsilon] \to_n 0$ for each 
$\varepsilon > 0$. In particular, $\rho(X^{(n)}) \toprobashort 0$ if 
$\rho(\bS) = 0$. 

\subsection{Random sparse sampling \`a la 
Erdős–Rényi of a continuous variance profile}
\label{erdos_reyni_subsection}
We now provide an example of a (semi-)sparse model covered by our result. In
this example, the sequence of matrices $(S^{(n)})$ will be random and
independent of the array $(W_{ij})_{ij\geq 1}$.  We shall show that
Assumptions~\ref{bnd-S} and~\ref{bnd-det} will be satisfied with high
probability (to be made precise below). In these conditions,
Theorem~\ref{th-main} will be obtained by conditioning on an appropriate event
which indicator is $S^{(n)}$--measurable. 

Let $\bS : [0,1]^2 \to [0,1]$ be a continuous function as in the previous
section. Assume for simplicity that the spectral radius $\rho(\bS)$ of $\bS$,
seen as an operator, is positive. Let us consider that $\rho(\bS) = 1$.  Our
variance profile matrix $S^{(n)}$ will be obtained by randomly sampling the
function $\bS$. Define the matrix 
\[
\uS^{(n)} = \begin{bmatrix} \uS^{(n)}_{ij} \end{bmatrix}_{i,j=1}^n 
 = \frac 1n \begin{bmatrix} \bS(i/n,j/n) 
  \end{bmatrix}_{i,j=1}^n.  
\]
Let $(K_n)$ be a sequence of positive numbers such that $K_n\to\infty$ and 
$K_n  = o(n)$.  For some large enough integer $n_0 > 0$, define the sequence 
of random matrices $(B^{(n)} = [B^{(n)}_{ij}] )_{n\geq n_0}$ as follows: 
\begin{align*}
    B^{(n)}_{ij} = \begin{cases}
        1 & \text{with probability } K_n \uS^{(n)}_{ij}  \\
        0 & \text{with probability } 1 - K_n \uS^{(n)}_{ij}, 
    \end{cases}
\end{align*}
and the random variables $\{ B^{(n)}_{ij} \}_{i,j\in[n]}$ are independent.  Let
\[
S^{(n)} = \frac{1}{K_n} B^{(n)}. 
\]
Trivially, $K_n \| S^{(n)} \|_\infty \leq 1$ for each elementary event, where
$\| \cdot \|_\infty$ is the max norm. Therefore,
Assumption~\ref{bnd-S}--\eqref{bnd-sij} is satisfied. Moreover, 
\begin{prop}
\label{ER} 
Assume that $K_n \geq \log n$. Then, there exists a constant $C_S  > 0$ such 
that $\limsup_n \vertiii{S^{(n)}} \leq C_S$ w.p.~1. 

Furthermore, for each $\varepsilon \in(0,1]$, it holds that 
\begin{equation} 
\label{q->S} 
 \min_{\gamma \in [0, 1 - \varepsilon]} 
  \det\left( I_n - \gamma S^{(n)} \right) 
   \xrightarrow[n\to\infty]{{\mathcal P}} 
  \det\left( \bI - (1-\varepsilon) \bS \right)  > 0,  
\end{equation} 
where $\xrightarrow{{\mathcal P}}$ is the convergence in probability. 
\end{prop}

\begin{cor}\label{confinement_cor_inhomogenous_erdos_reyni}
\label{ERnovap}
Assume that $K_n\geq\log n$. Then, 
$\PP\left[ \rho(X^{(n)}) \geq 1 + \varepsilon \right] \to_n 0$ for each 
$\varepsilon > 0$. 
\end{cor}

\begin{rem} 
If $W_{1,1}$ is a Rademacher random variable, then the matrix $X^{(n)}$ 
can be considered the centered adjacency matrix of a directed
inhomogeneous Erdős–Rényi graph. In recent years, there has been tremendous
attention on the spectrum of undirected inhomogeneous Erdős–Rényi models (see,
for example, \cite{avena2023limiting}, \cite{benaych2019largest},
\cite{chakrabarty2020eigenvalues}, and \cite{chakrabarty2021spectra}). A similar result to Corollary \ref{confinement_cor_inhomogenous_erdos_reyni}, is proven in Theorem 3.4 of \cite{benaych2020spectral}. In the case where $K_n=n^{o(1)}$, an analogue of Corollary \ref{confinement_cor_inhomogenous_erdos_reyni} is proven in Theorem 3 of \cite{coste2021simpler}, see also Remark 4. 
\end{rem}

\begin{rem}
The condition $K_n \geq \log n$ in the statement of Proposition~\ref{ER} is
required to obtain that $\limsup_n \vertiii{S^{(n)}} \leq C_S$ w.p.1., which
leads to Corollary~\ref{ERnovap} by the conditioning on the event $\cE_n$ that
we make in Section~\ref{prf-cor}.  We believe that this condition is not
necessary to obtain Theorem~\ref{th-main}.  Indeed, it is possible to obtain a
an analogue of Theorem~\ref{th-chpol} by including all the randomness of our
model within the matrix $X^{(n)}$ (without conditioning), and by simply taking
$\uS^{(n)}$ as the variance profile matrix.  We shall not develop this issue
here.  
\end{rem}

Let us provide a simple example where the condition $K_n \geq \log n$ is 
avoided while still using our Theorem~\ref{th-chpol} to obtain the spectral
radius confinement stated by Theorem~\ref{th-main}. 

\subsection{Random sparse sampling  
of a continuous variance profile with a fixed outer degree}
\label{d-reg-out} 
We still consider a operator represented by a continuous function $\bS :
[0,1]^2 \to [0,1]$ such that $\rho(\bS) = 1$.  Our variance profile matrix
$S^{(n)}$ is now obtained by randomly sampling the function $\bS$ as follows.
Let $(K_n)$ be a sequence of positive integers such that $K_n\to\infty$ and 
$K_n = o(n)$.  Let $\cI^{(n)}$ a random sub-set of $[n]$ which is uniformly
distributed among the $\binom{n}{K_n}$ sub-sets of $[n]$ with cardinality
$K_n$.  Let $\cI^{(n)}_1, \ldots, \cI^{(n)}_n$ be i.i.d.~subsets of $[n]$ such
that $\cI^{(n)}_1$ is equal to $\cI^{(n)}$ in distribution. Define the 
$\{ 0,1 \}^{n\times n}$--valued random matrix $R^{(n)} = [ R^{(n)}_{ij} ]$ as 
\[
R^{(n)}_{ij} = \begin{cases}
      1 & \text{if } j \in \cI^{(n)}_i \\ 
      0 & \text{if }  j \not\in \cI^{(n)}_i \\ 
 \end{cases}
\]
Finaly, let $S^{(n)} = [ S_{ij}^{(n)} ]$ be defined as 
\[
S_{ij}^{(n)} = \frac{n}{K_n} \uS^{(n)}_{ij} R^{(n)}_{ij} . 
\]
Trivially, $K_n \| S^{(n)} \|_\infty = \| \bS \|_\infty$ and 
$\vertiii{S^{(n)}} \leq \| \bS \|_\infty$, where $\| \bS \|_\infty$ is the 
norm of $\bS$ on $C([0,1])$. Regarding Assumption~\ref{bnd-det}, we have: 
\begin{prop}
\label{Reg}
For $\varepsilon \in (0,1]$, it holds that 
\[
 \min_{\gamma \in [0, 1 - \varepsilon]} 
  \det\left( I_n - \gamma S^{(n)} \right) 
   \xrightarrow[n\to\infty]{{\mathcal P}} 
  \det\left( \bI - (1-\varepsilon) \bS \right)  > 0. 
\]
\end{prop}

We close this section with a final remark. 
\begin{rem}
One can show that in the last three application examples,
Theorem~\ref{th-chpol} can be reformulated by stating that the sequence $(q_n)$
converges in distribution in the space $\HH$. We state without further comment
the expression of the limit in distribution $\bq\in\HH$, which is the same in
the three cases.  This limit reads: 
\[
\bq(z) = \sqrt{\det(\bI- z^2\E W^2_{1,1} \bS)} 
 \exp\left(- \sum_{k=1}^{\infty} z^k Z_k \sqrt{\frac{\tr \bS^k}{k}}\right),
 \quad z \in D(0,1), 
\]
where 
\begin{equation}
\label{trop} 
\tr \bS^k = \int_{[0,1]^k} \bS(x_1,x_2) \bS(x_2,x_3) \ldots \bS(x_k,x_1) 
   \prod_{i=1}^k dx_i.
\end{equation} 
\end{rem}

\section{Proof of Theorems~\ref{th-main} and~\ref{th-chpol}}
\label{sec-prfmain} 

In all the remainder, $C > 0$ is a generic constant independent of $n$ that can
change from a display to another. In the proofs, the superscript $^{(n)}$ such
as in $X^{(n)}$ will be often removed for notational simplicity.  Given a
matrix $M \in \C^{n\times n}$ and a set $I \subset [n]$ with cardinality $|I|$,
we denote as $M_I$ the $\C^{|I|\times |I|}$ sub-matrix of $M$ consisting of the
rows and columns which indices belong to $I$. Given a function $f : [0,1]^2 \to
\R$, we denote as $f\begin{pmatrix} x_1 & x_2 & \cdots & x_k \\ x_1 & x_2 &
\cdots & x_k \end{pmatrix}$ the $k\times k$ matrix which element $(i,j)$ is
$f(x_i, x_j)$. 

We start with the proof of Theorem~\ref{th-chpol}. 

\subsection{Proof of Theorem~\ref{th-chpol}: preliminary results on 
random holomorphic functions.} 

Before entering the proof of Theorem~\ref{th-chpol}, it will be useful to
recall first some basic results on the convergence in distribution of
random holomorphic functions. The reader is referred to \emph{e.g.}
\cite{shi-12} (see also \cite{bordenave2021convergence}) for more details on
this subject. In the three following propositions, we let $\epsilon \geq 0$ 
and we consider the space $\HH_\epsilon$ of holomorphic functions on the 
open disk $D(0,1-\epsilon)$. 
\begin{prop}
\label{cond-tight} 
Let $(f_n)$ be a sequence of random elements valued in $\HH_\epsilon$.  If, for
each compact set $\cK \subset D(0,1-\epsilon)$, the sequence of random
variables $(\max_{z\in\cK} | f_n(z) |)_n$ is tight, then, $(f_n)$ is tight.
For this condition to hold, it is enough that $\E |f_n(z) |^p \leq g(z)$  for
$p\geq 1$, where $g(z)$ is bounded on the compacts of $D(0,1-\epsilon)$. 
\end{prop}
\begin{prop}
Let $(f_n)$ be a tight sequence of random elements valued in $\HH_\epsilon$.
Denote as $f_n(z) = \sum_{k=0}^\infty a_k^{(n)} z^k$ the power series
representation of $f_n$ in $D(0,1-\epsilon)$. Assume that there exists a
sequence $a_0, a_1, \ldots$ of random variables such that for each positive
integer $m$, the $m$--tuple $(a_0^{(n)}, \ldots, a_m^{(n)})$ converges in
distribution to $(a_0, \ldots, a_m)$ as $n\to\infty$. Then, the function $f(z)
= \sum_{k=0}^\infty a_k z^k$ is well-defined as a random element valued in
$\HH_\epsilon$, and $(f_n)$ converges in distribution to $f$. 
\end{prop} 
This proposition can be easily modified
to obtain the following result, which is better suited to our context: 
\begin{prop}
\label{propcvg} 
Let $(f_n)$ and $(g_n)$ be two tight sequences of random elements valued in
$\HH_\epsilon$.  Denote as $f_n(z) = \sum_{k=0}^\infty a_k^{(n)} z^k$ and
$g_n(z) = \sum_{k=0}^\infty b_k^{(n)} z^k$ the power series representations of
$f_n$ and $g_n$ in $D(0,1-\epsilon)$ respectively. If for each fixed positive 
integer $m$, it holds that $(a_0^{(n)}, \ldots, a_m^{(n)}) \sim_n (b_0^{(n)}, 
 \ldots, b_m^{(n)})$, then $f_n \sim_n g_n$.  
\end{prop} 

To establish Theorem~\ref{th-chpol}, we start by writing $q_n(z)$ as 
\begin{equation}
\label{q-P} 
q_n(z)=\det(1-z X^{(n)})= 1+\sum_{k=1}^{n} (-z)^{k}P_{k}^{(n)}, 
\end{equation} 
where
\[
P_{k}^{(n)} =  \sum_{I \subset [n]: |I|=k} 
 \det X^{(n)}_I . 
\]
\subsection{Tightness of $(q_n)$} 
\label{subsec-tight} 
Our first result pertains to the tightness of the sequence of $\HHd$--valued 
random variables $(q_n)$: 
\begin{prop}
\label{tight_seq}
The sequence $(q_n)$ is tight. 
\end{prop}
To prove this proposition, we need the following result. 
\begin{lem} 
\label{perm} 
Let Assumption~\ref{bnd-det} hold true. Then, 
\[
\forall\varepsilon > 0, \ 
 \sup_n 
  \perm\left( I + ( 1 - \varepsilon ) S^{(n)} \right)  < \infty, 
\]
where $\perm M$ is the permanent of the matrix $M$. 
\end{lem} 
\begin{proof} 
Given $n > 0$, we identify the matrix $S^{(n)}$ with an integral kernel to
which we apply the Fredholm permanent theory developed in~\cite{ker-79}.  Our
kernel $\bS^{(n)} : [0,1) \times [0,1) \to \R_+$ is defined as $\bS^{(n)}(x,y)
= n S^{(n)}(i,j)$ when $(x,y) \in \left[ \frac{i-1}{n}, \frac in \right) \times
\left[ \frac{j-1}{n}, \frac jn \right), i,j \in[n]$. Following \cite{ker-79},
define the function $\bp^{(n)} : \C \to \C$ through the power series 
\[
\bp^{(n)}(w) = 1 + \sum_{k=1}^\infty \bp^{(n)}_k w^k, 
\]
with 
\[ 
\bp^{(n)}_k = \frac{1}{k !} 
 \int_0^1 \int_0^1 \cdots \int_0^1 
 \perm \bS^{(n)}\left(\begin{matrix} x_1 & x_2 & \cdots x_k \\  
  x_1 & x_2 & \cdots x_k \end{matrix}\right) dx_1 dx_2 \ldots dx_k . 
\]
Notice that $|\bp^{(n)}_k| \leq (n \| S^{(n)} \|_\infty)^k$.  Thus, the radius
of convergence $R^{(n)}$ of this series satisfies 
$R^{(n)} \geq 1/ (n \| S^{(n)} \|_\infty) > 0$, which shows that there exists a
centered open disk where $\bp^{(n)}(w)$ is well-defined and analytic. Let 
 $\bd: \C \to \C$ be given by the series
\[
\bd^{(n)}(w) = 1 + \sum_{k=1}^\infty (-1)^k \bd^{(n)}_k w^k, 
\]
with 
\[ 
\bd^{(n)}_k = \frac{1}{k !} 
 \int_0^1 \int_0^1 \cdots \int_0^1 
 \det \bS^{(n)}\left(\begin{matrix} x_1 & x_2 & \cdots x_k \\  
  x_1 & x_2 & \cdots x_k \end{matrix}\right) dx_1 dx_2 \ldots dx_k . 
\]
It is easy to see that for each $k \in [n]$, it holds that 
\[ 
\bd^{(n)}_k = 
 \int_0^1 dx_1 \int_0^{x_1} dx_2 \cdots \int_0^{x_{k-1}} dx_k 
 \det \bS^{(n)}\left(\begin{matrix} x_1 & x_2 & \cdots x_k \\  
  x_1 & x_2 & \cdots x_k \end{matrix}\right) 
 = \sum_{I\subset[n], |I|= k} \det S^{(n)}_I  , 
\]
and $\bd^{(n)}_k = 0$ for $k > n$. Therefore, $\bd^{(n)}(w)$ coincides 
with the reverse characteristic polynomial 
\[
\bd^{(n)}(w) = \det\left( I - w S^{(n)} \right). 
\]
Theorem 4.4~(a) of \cite{ker-79} states that 
\begin{equation}
\label{kersh} 
\bd^{(n)}(w) \bp^{(n)}(w) = 1 
\end{equation} 
for $w$ in the open disk of radius $R^{(n)}$. 
Since $\bp^{(n)}_k \geq 0$ for each $k$, the spectral radius
$R^{(n)}$ is a singular point of $\bp^{(n)}(w)$  (see
\cite[Fact~7.21]{tit-livre39}), and thus, it is a zero of $\bd^{(n)}(w)$ by the
previous identity.  By Assumption~\eqref{bnd-det}, we then obtain that 
$\liminf_n R^{(n)} \geq 1$, and by Identity~\eqref{kersh} again, it holds that
\[
\forall\varepsilon > 0, \quad 
  \sup_n \bp^{(n)}(1 - \varepsilon) < \infty. 
\]
Let $p^{(n)}(w) = \perm\left( I + w S^{(n)} \right)$. As is well known 
(see, \emph{e.g.}, \cite[Th.~1.4]{min-livre78}), $p^{(n)}(w) = 1 + 
\sum_{k=1}^n p^{(n)}_k w^k$ with  
\[
p^{(n)}_k = \sum_{I\subset[n], |I|= k} \perm S^{(n)}_I  
   \quad \text{for } k \in [n] .  
\]
Writing 
\[ 
\bp^{(n)}_k = 
 \int_0^1 dx_1 \int_0^{x_1} dx_2 \cdots \int_0^{x_{k-1}} dx_k 
 \perm \bS^{(n)}\left(\begin{matrix} x_1 & x_2 & \cdots x_k \\  
  x_1 & x_2 & \cdots x_k \end{matrix}\right) , 
\]
it is easy to see that $p^{(n)}_k \leq \bp^{(n)}_k$ for each $k\in[n]$, 
therefore, 
\[
\forall\varepsilon > 0, \quad 
  \sup_n p^{(n)}( 1 - \varepsilon) < \infty, 
\]
which is the required result. 
\end{proof} 

\begin{proof}[Proof of Proposition~\ref{tight_seq}] 
To prove our proposition, we bound $\E |q_n(z)|^2$ and use
Proposition~\ref{cond-tight}. 

Denoting as $\mathfrak S_I$ the group of permutations over a set 
$I \subset [n]$, and $\sign(\sigma)$ the signature of a permutation $\sigma$, 
we first observe that 
\[
 \E \det X_I = \sum_{\sigma \in \mathfrak S_I} 
  \sign(\sigma)\E \prod_{i\in I} X_{i, \sigma(i)}=0 , 
\]
since the entries of $X$ are centered. 

Similarly for  $J, I \subset [n]$ such that $I \neq J$ it is true that 
\[
\E \det X_I \, \overline{\det X_J}=0 . 
\]
Lastly, for any $I \subset [n]$, 
\[
\E|\det X_I |^2= \E \det X_I \, \overline{\det X_I}
 = \sum_{\sigma \in \mathfrak S_I} \prod_{i\in I} 
   s_{i,\sigma(i)}\E |W_{i, \sigma(i)}|^2= \perm S_I.  
\]
We therefore have 
\[ 
   \E |q_n(z)|^2 = \E \left| 1 + 
  \sum_{k=1}^n (-z)^k \sum_{I \subseteq [n]:|I|=k} 
          \det X_I \right|^2 = \perm\left( I + |z^2| S^{(n)} \right), 
  \]
which is bounded by the previous lemma on the compacts of $D(0,1)$. 
\end{proof}

\subsection{Asymptotics of the finite-dimensional distributions when 
$W_{11}$ is bounded} 
\label{subsec-finidim} 
Having established the tightness of $(q_n)$, it remains to examine the
distributional large--$n$ properties of the random vector $(P_1^{(n)}, \ldots,
P_{k}^{(n)})$ for each fixed integer $k > 0$, and apply
Proposition~\ref{propcvg} above.  To this end, we temporarily assume that the
random variables $W_{ij}$ are bounded by a constant. We also rely on the fact
that in order to study the distribution of $(P_1^{(n)}, \ldots, P_{k}^{(n)})$,
it is enough to study the distribution of $(\tr X^{(n)},\ldots, \tr
(X^{(n)})^k)$ for large $n$, a much easier task. Specifically, for $z\in\C$, 
the series $\sum_{k=1}^{\infty} (z^{k}/k) (X^{(n)})^{k}$ is well-defined for 
$|z|$ small enough, and we can express $q_n(z)$ as 
\begin{equation} 
\label{q_=tr}
   q_n(z)= \exp \left(
 -\sum_{k=1}^{\infty}\tr((X^{(n)})^{k}) \frac{z^{k}}{k} \right)
\end{equation} 
for $|z|$ small enough. Recalling the identity~\eqref{q-P}, we obtain that the 
$k$--tuple $(P_{1}^{(n)}, \ldots, P^{(n)}_k)$ is a polynomial function of $(\tr
X^{(n)},\ldots,\tr (X^{(n)})^k)$ independent of the dimension $n$, by an
expansion to a power series on both sides of \eqref{q_=tr} and by examining at
the first $k$ terms of the expansion.
Thus, we end up that in order to prove the asymptotic equivalence
in~\eqref{asymptotic_equivalence_of_qn}, it is sufficient to examine the
large-$n$ distributions of the vectors $(\tr X^{(n)},\ldots,\tr (X^{(n)})^k)$
for any integer $k > 0$. We shall analyze the distributions of these vectors 
with the help of the moment method, which explains why the assumed boundedness 
of the $W_{ij}$ is important in our proof. Of course, the finiteness of all 
their moments would have been enough. 
\begin{prop}
\label{fin-dim}
Assume that the random variables $W_{ij}$ are bounded by a constant. Consider
the sequence of independent complex-valued Gaussian random variables
$(Z_\ell)_{\ell\geq 1}$ defined in the statement of Theorem~\ref{th-chpol}. For
each integers $n, \ell > 0$, define $\bm_\ell^{(n)}$ as 
\[
\bm_\ell^{(n)} = 
\left\{\begin{array}{cl} 
 \left(\E W_{11}^2\right)^{\ell/2} \tr((S^{(n)})^{\ell/2}) &\text{if } 
     \ell \text{ is even, } \\
  0  &\text{if } \ell \text{ is odd.} \\
 \end{array}\right. 
\]
Then, for each fixed integer $k > 0$, the asymptotic equivalence 
\begin{equation} 
\label{as-eq} 
\left( \tr X^{(n)}, \ldots, \tr (X^{(n)})^{k} \right) \sim_n  
\left( \sqrt{\tr S^{(n)}} Z_{1} + \bm_{1}^{(n)}, \ldots, 
 \sqrt{k \tr (S^{(n)})^{k}} Z_{k} + \bm_{k}^{(n)} \right) 
\end{equation} 
holds true. 
\end{prop} 
Most of the remainder of this section is devoted to the proof of this 
proposition. We start with a simple lemma. 
\begin{lem}
\label{trSk}
Let Assumptions~\ref{bnd-S} and \ref{bnd-det} hold true. Then 
\[
\forall k > 0, \exists C > 0, \ \tr S^k \leq C \ \text{and} \ 
  \| S^k \|_\infty \leq C / K_n . 
\]
\end{lem} 
\begin{proof}
Using Assumption~\ref{bnd-det} and recalling the development~\eqref{dev-det},
we obtain the first bound by setting, \emph{e.g.}, $\gamma = 1/2$. 

Assumption~\ref{bnd-S}--\eqref{bnd-sij} asserts that $\| S \|_\infty \leq C /
K_n$, thus, the second bound is effective for $k=1$.  Assume without generality
loss that $\vertiii{S} \leq C$ from Assumption~\ref{bnd-S}--\eqref{rowsum}.
For $k > 1$, we have $\| S^k \|_\infty \leq \vertiii{S} \| S^{k-1} \|_\infty
\leq \cdots \leq \vertiii{S}^{k-1} \| S \|_\infty \leq C / K_n$.  
\end{proof} 
Given a $k$--tuple $\bI = (i_1,\ldots, i_k) \subset[n]^k$, we write 
\[
X_{\bI} = X_{i_1i_2} X_{i_2 i_3} \ldots X_{i_{k-1} i_k} X_{i_k i_1} . 
\]
As in \cite{bordenave2021convergence}, we write 
$\tr (X^k) = \sum_{\bI\in[n]^k} X_{\bI}$ as 
\[
\tr ((X^{(n)})^k) = R_k^{(n)} + Q_k^{(n)} , 
\]
where, denoting as $\cD_k$ the sub-set of $[n]^k$ defined as 
\[
\cD_k = \{ (i_1,\ldots, i_k) \in [n]^k \ : \ \forall j\neq\ell \in [k], 
  i_j \neq i_\ell \}, 
\]
we set 
\[
R_k = \sum_{\bI \in \cD_k} X_{\bI}, \quad \text{and} \quad 
Q_k = \sum_{\bI \in [n]^k\setminus\cD_k} X_{\bI}. 
\]
It is obvious that $\E R_k = 0$. The analogues of $R_k$ and $Q_k$ are called in
\cite{bordenave2021convergence} the ``random term'' and the ``deterministic
term'' respectively. We shall deal with these two terms separately. The 
following two lemmas are proven in Section~\ref{prf-lm-findim} below. 
\begin{lem}
\label{Rk}  
Let the $m$--tuple $(k_1,\ldots, k_m)$ be as in the statement of
Proposition~\ref{fin-dim}. Given $x \in \C$, use the notation $x^s = x$ when $s
= \cdot$ and $x^s = \bar x$ when $s = *$.  
Let $s_1, \ldots, s_m \in \{ \cdot, * \}$.  

If $m$ is even, and if there exists at least one partition $P$ of $[m]$ into
pairs such that $k_\ell = k_{\ell'}$ if $\ell$ and $\ell'$ are a pair (notation
$\{\ell,\ell'\} \in P$), then, 
\[
\E\left[ R_{k_1}^{s_1} R_{k_2}^{s_2} \ldots R_{k_m}^{s_m} \right] 
- \sum_{P\in\bP} \prod_{\{\ell,\ell'\} \in P} 
  \left( k_\ell \left( \E W_{11}^{s_\ell} W_{11}^{s_{\ell'}} \right)^{k_\ell} 
 \tr S^{k_\ell}  \right) \xrightarrow[n\to\infty]{} 0 , 
\]
where $\bP$ is the set of such partitions. Otherwise, 
\[
\E\left[ R_{k_1}^{s_1} R_{k_2}^{s_2} \ldots R_{k_m}^{s_m} \right] 
  \xrightarrow[n\to\infty]{} 0 . 
\]
\end{lem} 
\begin{lem}
\label{Qk} 
It holds that 
\[
Q_k - \bm_k \toprobalong 0 . 
\]
\end{lem} 

\subsubsection*{Proof of Proposition~\ref{fin-dim}.} 
By Lemma~\ref{trSk}, for each sequence $(n)$ of integers, there exists a
sub-sequence such that for every integer $\ell > 0$, $\tr (S^{(n)})^\ell$
converges to some real number $\bss_\ell$ along this sub-sequence.  Fix an
integer $k > 0$.  Lemma~\ref{Rk} along with the Isserlis/Wick theorem show that
$(R^{(n)}_1,\ldots, R^{(n)}_k)$ converges in distribution along this
sub-sequence to $(\sqrt{\bss_1} Z_1, \ldots \sqrt{k \bss_k} Z_k)$.  By
Lemma~\ref{Qk}, for $\ell\in[k]$, $Q^{(n)}_\ell$ converges in probability along
this sub-sequence to $\left(\E W_{11}^2\right)^{\ell/2} \bss_{\ell/2}$ if 
$\ell$ is even and to zero if $\ell$ is odd. The result stated by 
Proposition~\ref{fin-dim} follows. 

\subsubsection*{Proofs of Lemmas~\ref{Rk} and \ref{Qk}} 
\label{prf-lm-findim} 
The following preliminary result will be needed. 
\begin{lem}
\label{errS} 
Let Assumptions~\ref{bnd-S} and \ref{bnd-det} hold true. Let $k_1, \ldots, 
k_m$ be positive integers, and write $k = k_1+\cdots + k_m$. Decomposing a 
$k$--tuple $\bI \in [n]^k$ as $\bI = (\bI_1, \ldots, \bI_m)$ where 
$\bI_j \in [n]^{k_j}$, it holds that there exists $C > 0$ such that 
\[
0 \leq \tr S^{k_1} \ldots \tr S^{k_m} - 
 \sum_{\bI \in \cD_k} S_{\bI_1} \ldots S_{\bI_m}  \leq C / K_n. 
\]
\end{lem} 
\begin{proof} 
Observe first that $\tr S^{k_1} \ldots \tr S^{k_m} = \sum_{\bI \in [n]^k}
S_{\bI_1} \ldots S_{\bI_m}$. If $k = 1$, the result is trivial. Assume not.
The indicator function $\1_{\cD_k}(\bI)$ with $\bI = (i_1,\ldots, i_k)$ can be
encoded into the product of the $k(k-1)/2$ indicators of the type 

$\1_{i_j \neq
i_{\ell}}$ for $j \neq \ell \in [k]$. Let us order the constraints $i_j \neq
i_{\ell}$ in some way from $1$ to $k(k-1)/2$, and let us write 
$\1^{(m)}(\bI)$
as the product of the indicators on the first $m$ constraints, so that
$\1^{(k(k-1)/2)}(\bI) = \1_{\cD_k}(\bI)$.  Writing $\1^{(0)} \equiv 1$, we 
have  
\[
\sum_{\bI \in [n]^k} S_{\bI_1}\ldots S_{\bI_m} 
  - \sum_{\bI \in \cD_k}  S_{\bI_1}\ldots S_{\bI_m} = 
\sum_{m=0}^{k(k-1)/2-1} \sum_{\bI \in [n]^k} 
   (\1^{(m)}(\bI) - \1^{(m+1)}(\bI))  S_{\bI_1}\ldots S_{\bI_m} .  
\]
Write 
$\1^{(m)}(\bI) - \1^{(m+1)}(\bI) = \1^{(m)}(\bI) \1_{i_j = i_{\ell}}$
for some $j\neq\ell\in[k]$. Assuming, \emph{e.g.}, $k_1\geq 2$, $j=1$, and 
$\ell \leq k_1$, we obtain 
\begin{align*} 
 &\sum_{\bI \in [n]^k} (\1^{(m)}(\bI) - \1^{(m+1)}(\bI)) 
  S_{\bI_1} \ldots S_{\bI_m} 
 \leq \sum_{\bI \in [n]^k} \1_{i_1 = i_{\ell}} S_{\bI_1} \ldots S_{\bI_m} \\
 &= \sum_{i_1,\ldots,i_{\ell-1},i_{\ell+1},\ldots, i_{k_1}}
   s_{i_1i_2} \ldots s_{i_{\ell-1} i_1} \ 
  s_{i_1 i_{\ell+1}} \ldots s_{i_{k_1} i_1} \ \tr S^{k_2} \ldots \tr S^{k_m} \\
&\leq C \sum_{i_1} \begin{bmatrix} S^{\ell-1} \end{bmatrix}_{i_1i_1}  
  \begin{bmatrix} S^{k_1-\ell+1} \end{bmatrix}_{i_1i_1}  
\leq \frac{C}{K_n} \tr S^{k_1-\ell+1} \leq \frac{C}{K_n} 
\end{align*} 
by Lemma~\ref{trSk}. The cases where the indices $j$ and $\ell$ belong to two
different tuples $\bI_r$ are treated similarly. 
\end{proof}

\begin{proof}[Proof of Lemma \ref{Rk}]  
We deal with the expression 
\[
R_{k_1}^{s_1} R_{k_2}^{s_2} \ldots R_{k_m}^{s_m} = 
\sum_{(\bI_{1}, \ldots, \bI_{m}) \in \cD_{k_1} \times \cdots \times 
  \cD_{k_m}} X_{\bI_{1}}^{s_1} \ldots X_{\bI_{m}}^{s_m}. 
\]
We introduce some new notation. We let $k = k_1 + \cdots k_m$, and we write 
\[
\bI = (\bI_{1}, \ldots, \bI_{m}) = 
 \left( (i^1_1,\ldots, i^1_{k_1}), \ldots, (i^m_1,\ldots,i^m_{k_m})\right) 
  \in [n]^k. 
\]
In what follows, it is always meant that $i^\ell_{k_\ell+j} = i^\ell_j$ 
for $j=1,\ldots,k_\ell-1$. 

When $k = 1$ ($=m$), it is obvious that $\E R^{s_1}_1 = 0$, thus the lemma is 
true. Assume that $k > 1$, and define the two sets $\cA, \cB \subset 
 \cD_{k_1} \times \cdots \times \cD_{k_m}$ as 
\begin{align*}
\cA = \Bigl\{ \bI \in \cD_{k_1} \times \cdots \times \cD_{k_m}  \ : \   
   &\text{each couple } (i^\ell_j, i^\ell_{j+1}) 
      \text{ appears exactly twice in } \bI ,  \\ 
 &\text{each index } i^\ell_j  
  \text{ appears exactly twice in } \bI \Bigr\} , \\ 
\cB = \Bigl\{ \bI \in \cD_{k_1} \times \cdots \times \cD_{k_m} \ : \ 
   &\text{each couple } (i^\ell_j, i^\ell_{j+1}) 
      \text{ appears at least twice in } \bI ,  \\ 
 &\text{there exists an index } i^\ell_j  
  \text{ that
appears three times at least in } \bI \Bigr\}.  
\end{align*}
Since the elements of the matrix $X$ are centered, $\E X_{\bI_{1}}^{s_1} \ldots
X_{\bI_{m}}^{s_m}$ is equal to zero if there exists a couple $(i^\ell_j,
i^\ell_{j+1})$ that appears only once within $\bI$. This implies that 
\begin{equation}
\label{piR} 
\E R_{k_1}^{s_1} R_{k_2}^{s_2} \ldots R_{k_m}^{s_m} = 
\sum_{\bI \in \cA} \E X_{\bI_{1}}^{s_1} \ldots X_{\bI_{m}}^{s_m} 
 + 
\sum_{\bI \in \cB} \E X_{\bI_{1}}^{s_1} \ldots X_{\bI_{m}}^{s_m} 
\end{equation} 
We now show that 
\begin{equation}
\label{deg3}
\left| \sum_{\bI \in \cB} \E X_{\bI_{1}}^{s_1} \ldots
X_{\bI_{m}}^{s_m} \right| \leq \frac{C}{\sqrt{K_n}}. 
\end{equation}
There are $2^{k(k-1)/2}$ ways of constructing an indicator function on  $[n]^k$
defined as a product of indicators of the type $\1_{i^\ell_i = i^{\ell'}_{i'}}$
and indicators of the type $\1_{i^\ell_i \neq i^{\ell'}_{i'}}$, where this
product involves all the $k(k-1)/2$ sets of the type $\{(\ell,i), (\ell',i')
\}$ with cardinality $2$. There is a sub-set of these functions that completely 
describes the set $\cB$ in the sense that we can write
\[
\1_{\cB}(\bI) = \sum_{r=1}^{C_B} f_r(\bI)
\]
where the functions $f_r$ are chosen appropriately in the family that we just
defined, and where $C_B = C_B(k_1,\ldots, k_m)$ is the number of these 
functions. To establish~\eqref{deg3}, we show that for each $r \in [C_B]$, it 
holds that 
$\sum_{\bI \in [n]^k} | \E X_{\bI_{1}}^{s_1} \ldots X_{\bI_{m}}^{s_m} | 
  f_r(\bI) \leq C / \sqrt{K_n}$. 
Relying on the boundedness of the elements of $X$, we write 
\begin{equation} 
\label{SI} 
\sum_{\bI \in [n]^k} | \E X_{\bI_{1}}^{s_1} \ldots X_{\bI_{m}}^{s_m} | 
  f_r(\bI) 
\leq C \sum_{\bI \in [n]^k} \sqrt{S_{\bI_{1}} \ldots S_{\bI_{m}}}  f_r(\bI) . 
\end{equation} 
To deal with this expression, we need to introduce some new notations. Given
two $d$--tuples  $\bJ = (j_1,\ldots,j_d) \in [n]^d$ and $\ba = (a_1,\ldots,a_d)$
with $a_i \in \{ 1, 3/2, 2, 5/2, \ldots \}$, we write $|\bJ| = |\ba| = d$, 
\[
S_{\bJ}^{\ba} = s^{a_1}_{j_1,j_2} s^{a_2}_{j_2,j_3} \ldots 
  s^{a_{d}}_{j_{d-1},j_d} s^{a_{d}}_{j_d,j_1} 
\quad \text{and} \quad 
S_{\obJ}^{\ba} = 
  s^{a_1}_{j_1,j_2} s^{a_2}_{j_2,j_3} \ldots s^{a_{d}}_{j_{d-1},j_d} . 
\]
Of course, $S_{\bJ} = S_{\bJ}^{\bs 1_d}$. We also write 
$S_{\obJ} = S_{\obJ}^{\bs 1_d}$. 
With these notations, the right-hand side of~\eqref{SI} can be re-expressed as
follows.  By merging all the couples $(i^\ell_j, i^\ell_{j+1})$ that are forced
to be identical in the encoding by $f_r$, and keeping after a merger the couple
$(i^\ell_j, i^\ell_{j+1})$ with the smallest value of $\ell$, we can observe
after a possible index renumbering that there exists:
\begin{itemize}
\item An integer $p > 0$, tuples $\bJ_1, \ldots, \bJ_p$, and 
$\ba_1, \ldots, \ba_p$ such that $|\bJ_\ell | = | \ba_\ell | \in 
\{k_1,\ldots, k_m \}$ for $\ell\in[p]$, 
\item An integer $q \geq 0$, and, when $q > 0$, tuples
$\bJ_{p+1}, \ldots, \bJ_{p+q}$ and   
$\ba_{p+1}, \ldots, \ba_{p+q}$ with $|\bJ_{p+\ell} | = | \ba_{p+\ell} |$, 
\end{itemize} 
such that 
\begin{align}
&\sum_{\bI \in [n]^k} \sqrt{S_{\bI_{1}} \ldots S_{\bI_{m}}}  f_r(\bI) 
 \nonumber \\
 &=  \sum_{\bJ_1,\ldots,\bJ_p,\bJ_{p+1},\ldots,\bJ_{p+q}} 
 S^{\ba_1}_{\bJ_1} \cdots S^{\ba_p}_{\bJ_p} 
 S^{\ba_{p+1}}_{\obJ_{p+1}} \cdots S^{\ba_{p+q}}_{\obJ_{p+q}} 
 g(\bJ_1,\ldots,\bJ_{p+q}) :=  \chi, 
\label{err} 
\end{align} 
an expression that we now explain.  The function $g(\bJ_1,\ldots,\bJ_{p+q})$ is
a product of indicators that encodes the residual constraints after the merger
of the couples.  Let us see an initial tuple $\bI_\ell$ as constituting a
cycle $i^\ell_1 \to i^\ell_2 \to \cdots \to i^\ell_{k_\ell} \to i^\ell_1$.
After the merger, $\bI_1$ gives rise to $\bJ_1$. The cycle in $\bJ_1$ is not
broken when performing the sum in~\eqref{err}, since the couple $(i^\ell_j,
i^\ell_{j+1})$ with the smallest value of $\ell$ is kept after a merger.
Pursuing, the tuples $\bJ_2,\ldots,\bJ_p$ give rise to unbroken cycles in
the expression~\eqref{err}. The other $\bJ_{p+\ell}$'s, when they exist,
correspond to broken cycles. Let us consider $\bJ_{p+1}$. An important
feature of this tuple is that its extremities are connected to the
$\bJ_\ell$'s for $\ell \in [p]$ by the merger procedure. In other words,
writing in the remainder $\bJ_{\ell} = ( j^{\ell}_1, \ldots,
j^{\ell}_{|\bJ_{\ell}|})$,   there exists within the function $g$ a product of
indicators of the type $\1_{j^{p+1}_1 = \times} \1_{j^{p+1}_{|\bJ_{p+1}|} =
\times'}$ where the indices $\times$ and $\times'$ belong to the $\bJ_\ell$ for
$\ell\in[p]$. Pursuing this process, the extremities of the tuple
$\bJ_{p+q}$ are connected to the $\bJ_\ell$'s for $\ell \in [p+q-1]$.  

We now use these observations to bound $\chi$.  We shall repeatedly use the
bound $S^{\ba_\ell}_{\bJ_\ell} \leq C S_{\bJ_\ell}$ and
$S^{\ba_\ell}_{\obJ_\ell} \leq C S_{\obJ_\ell}$ due to $\| S \|_\infty \leq 1$
for all large $n$. According to the form of the function $f_r$, at least one of
the three following situations occurs:  
\begin{itemize}
\item $q > 0$. Assume for the sake of example that $q = 1$, and write
$j^{p+1}_1 = j^\ell_i$ and $j^{p+1}_{|\bJ_{p+1}|} = j^{\ell'}_{i'}$ where
$j^\ell_i$ and $j^{\ell'}_{i'}$ are found in $\bJ_1,\ldots, \bJ_p$. Using
Lemma~\ref{trSk}, we have
\begin{align*} 
\chi &\leq C \sum_{\bJ_1,\ldots,\bJ_p,\bJ_{p+1}} 
 S_{\bJ_1} \cdots S_{\bJ_p} 
 S_{\obJ_{p+1}} \1_{j^{p+1}_1 = j^\ell_i} 
   \1_{j^{p+1}_{|\bJ_{p+1}|} = j^{\ell'}_{i'}} 
  \\
 &=C \sum_{\bJ_1,\ldots,\bJ_p} 
 S_{\bJ_1} \cdots S_{\bJ_p} 
  \begin{bmatrix} S^{|\bJ_{p+1}|-1} \end{bmatrix}_{j^\ell_i j^{\ell'}_{i'}} \leq \frac{C}{K_n} \tr S^{|\bJ_1|} \ldots \tr S^{|\bJ_p|} 
 \leq \frac{C}{K_n}. 
\end{align*}
For general $q$, we get the bound $C/K_n^q$ by iterating this argument 
backwards, starting with $\bJ_{p+q}$. 

\item $q = 0$, and there exists an exponent within the $\ba_\ell$ that is 
 $\geq 3/2$. Here it is easy to observe that $\chi \leq C / \sqrt{K_n}$. 

\item $q = 0$ and all the vectors $\ba_\ell$ are made of ones. 
Since there is an index that appears at least three times in $f_r(\bI)$ by the 
definition of $\cB$, there is an index that appears at least two times 
in the $\bJ_\ell$'s. Say this index is $j^1_1$ with $j^1_1 = j^2_1$. Then, 
\begin{align*}
\chi &\leq \sum_{\bJ_1,\ldots,\bJ_p} 
 S_{\bJ_1} \cdots S_{\bJ_p} \1_{j^1_1 = j^2_1} 
 \leq C \sum_{\bJ_1, \bJ_2} S_{\bJ_1} S_{\bJ_2} \1_{j^1_1 = j^2_1}  \\ 
 &\leq C \sum_{j_1} \begin{bmatrix} S^{|\bJ_1|-1} \end{bmatrix}_{j_1j_1} 
 \begin{bmatrix} S^{|\bJ_2|-1} \end{bmatrix}_{j_1j_1}  
 \leq \frac{C}{K_n}. 
\end{align*}
\end{itemize}
This establishes Inequality~\eqref{deg3}. 

Getting back to~\eqref{piR}, we now deal with the term $\sum_{\bI \in \cA} \E
X_{\bI_{1}}^{s_1} \ldots X_{\bI_{m}}^{s_m}$.  Here, one can check that a
necessary condition for $\cA$ to be nonempty is that $m$ is even and the
tuples $\bI_1$, ..., $\bI_m$ can be grouped into pairs of equal length. 
In other words, the set $\bP$ of pair partitions $P$ of $[m]$ as specified in 
the statement of the lemma is not empty. The case being, we have
\[
\1_{\cA}(\bI) = \sum_{P\in\bP} h_P(\bI) 
 \prod_{\{\ell,\ell'\} \in P}  
 \left( \prod_{j=1}^{k_\ell} \1_{i^\ell_j = i^{\ell'}_j} 
 + \prod_{j=1}^{k_\ell} \1_{i^\ell_{j+1} = i^{\ell'}_j} + \cdots + 
  \prod_{j=1}^{k_\ell} \1_{i^\ell_{j+k_\ell-1} = i^{\ell'}_j} \right) , 
\]
where $h_P(\bI) \in \{0,1\}$ forces the indices $i^\ell_1, \ldots,
i^\ell_{k_\ell}$ within the pair $\{\ell,\ell'\}\in P$ to be different, and to
be different from the indices within all the other pairs. To better understand
the previous formula, let us see once again the tuples $\bI_\ell$ as
cycles. When $\{ \ell, \ell' \} \in P$, we need to make the cycles
associated to $\bI_\ell$ and $\bI_{\ell'}$ coincide, and there are $k_\ell$
ways to do this. This corresponds to the sum of products within the parenthesis
of the last display. 

With this identity, we have 
\begin{align*} 
& \sum_{\bI \in \cA} \E X_{\bI_{1}}^{s_1} \ldots X_{\bI_{m}}^{s_m} \\ 
&= \sum_{P\in\bP} \sum_{\bI\in[n]^k} 
h_P(\bI) 
 \prod_{\{\ell,\ell'\} \in P}  
 \left( \prod_{j=1}^{k_\ell} \1_{i^\ell_j = i^{\ell'}_j} 
 + \prod_{j=1}^{k_\ell} \1_{i^\ell_{j+1} = i^{\ell'}_j} + \cdots + 
  \prod_{j=1}^{k_\ell} \1_{i^\ell_{j+k_\ell-1} = i^{\ell'}_j} \right) 
\E X_{\bI_{1}}^{s_1} \ldots X_{\bI_{m}}^{s_m} \\
&= \sum_{P\in\bP} \left( \prod_{\{\ell,\ell'\} \in P} 
   k_\ell \left( \E W_{11}^{s_\ell} W_{11}^{s_{\ell'}} \right)^{k_\ell} \right) 
\sum_{\neq} \prod_{\{\ell,\ell'\} \in P} S_{\bI_\ell} , 
\end{align*}
where $\sum_{\neq}$ is the sum over all the $m/2$ tuples $\bI_\ell$ such
that $\{ \ell,\ell'\} \in P$, with the constraint that all the indices that
belong to these tuples are different. Thanks to Lemma~\ref{errS}, we 
obtain that 
\[
\sum_{\bI \in \cA} \E X_{\bI_{1}}^{s_1} \ldots X_{\bI_{m}}^{s_m} \\ 
= \sum_{P\in\bP} \prod_{\{\ell,\ell'\} \in P} 
  \left( k_\ell \left( \E W_{11}^{s_\ell} W_{11}^{s_{\ell'}} \right)^{k_\ell} 
 \tr S^{k_\ell}  \right)  + \varepsilon , 
\] 
where $|\varepsilon| \leq C / K_n$. Recalling~\eqref{deg3}, our lemma is 
proven.
\end{proof} 

\begin{proof}[Proof of Lemma~\ref{Qk}]  

We first evaluate the asymptotics of $\E Q_k$. Assuming $k$ is even, let us
focus on the case where the indices of $\bI = (i_1,\ldots, i_k)$ satisfy the
constraints $i_j = i_{k/2+j}$ for all $j\in[k/2]$ and $|\{ i_1, \ldots,
i_{k/2} \}| = k/2$, generating the double cycle $i_1 \to \cdots i_{k/2} \to i_1
\to \cdots i_{k/2} \to i_1$. The expectation of the sum over the $\bI$ with
these constraints is $\sum_{\bJ \in \cD_{k/2}} \left(\E W_{11}^2 \right)^{k/2}
S_{\bJ} = \left(\E W_{11}^2 \right)^{k/2} \tr S^{k/2} + \cO(1/K_n) = \bm_k +
\cO(1/K_n)$ thanks to Lemma~\ref{errS}. When they exist, all other possibly
non-zero contributions to $\E Q_k$, including $k$ being odd, correspond to the
couples $(i_j, i_{j+1})$ appearing two times at least and one index appearing
three times at least. By treating these cases similarly to what we did for the
term $\sum_{\bI \in \cB}\cdots$ in the previous lemma, we can show that these
cases are bounded by $C / \sqrt{K_n}$.  We thus have 
\[
\E Q_k - \bm_k \xrightarrow[n\to\infty]{} 0 .
\] 
To establish the result of the lemma, we now show that the variance 
$\var(Q_k)$ of $Q_k$ converges to zero. Write 
\[
\var Q_k = \sum_{\bI_1 = (i_1^1,\ldots,i_k^1), \bI_2 = (i_1^2,\ldots,i_k^2) 
  \in [n]^k\setminus\cD_k} 
  \E (X_{\bI_1} - \E X_{\bI_1})(X_{\bI_2} - \E X_{\bI_2}) . 
\] 
We observe here that 
\[
\left| \E (X_{\bI_1} - \E X_{\bI_1})(X_{\bI_2} - \E X_{\bI_2}) \right| 
 \leq C \sqrt{s_{i^1_1 i^1_2} \ldots s_{i^1_k i^1_1}} 
  \sqrt{s_{i^2_1 i^2_2} \ldots s_{i^2_k i^2_1}}. 
\]
Moreover, the left hand side is equal to zero unless every couple $(i^\ell_{j},
i^\ell_{j+1})$ appears at least twice in $\bI = (\bI_1, \bI_2)$, and at least
one of these couples is common to $\bI_1$ and $\bI_2$.  In a manner similar to
what we did in the proof of the previous lemma, these constraints can be
encoded into functions that have the form of products of the type $\1_{i^\ell_j
= i^{\ell'}_{j'}}$. The number of such functions does not depend on $n$. Fixing
one of these functions, we merge the identical couples within $\bI_1$ and
within $\bI_2$, keeping the lowest indices, and then we merge what remains in
$\bI_1$ and $\bI_2$, keeping the lowest exponent.  By doing so, we get an
expression similar to \eqref{err} in the proof of the previous lemma. Re-using
the notations of that proof and repeating the argument there, the case where 
$q > 0$ and the case $q = 0$ with an exponent $\geq 3/2$ have negligible
contributions.  Let us deal with the case where $q = 0$ and where all the
exponents are equal to $1$.  In this case, each couple $(i^\ell_j,
i^\ell_{j+1})$ appears exactly twice, and we recall that there is one couple
common to $\bI_1$ and $\bI_2$.  Also observe that since $\bI_1, \bI_2 \in
[n]^k\setminus\cD_k$, there is at least one index repetition within $\bI_1$ and
$\bI_2$. In these conditions, one can check that the only available
possibilities are of the form $\sum_{\bJ_1,\bJ_2} S_{\bJ_1} S_{\bJ_2}
\1_\times$, where $\1_\times$ links an index in $\bJ_1$ to an index in $\bJ_2$.
This leads to a negligible contribution. Lemma~\ref{Qk} is proven. 
\end{proof}

\subsection{Proof of Theorem~\ref{th-chpol} when $W_{11}$ is bounded} 
We begin by establishing the properties of the functions $\kappa_n$ and $F_n$
provided in the statement of Theorem~\ref{th-chpol}.  Recalling that $\limsup
\rho(S^{(n)}) \leq 1$, the function $\kappa_n$ is well-defined as an element of
$\HHd$ for all large $n$. Moreover, using Assumption~~\eqref{bnd-det} and
recalling the development~\eqref{dev-det}, we obtain that for all small 
$\varepsilon > 0$,  
\[
\max_{\gamma\in[0,1-\varepsilon]} \sum_{k=1}^\infty
  \gamma^k \frac{\tr (S^{(n)})^k}{k}  < \infty. 
\]
Therefore, for each compact $\cK \subset D(0,1)$, it holds that 
\begin{align*}
\limsup_n\max_{z\in\cK} 
 \left| \log\det( I_n - z^2 \E W_{11}^2 S^{(n)}) \right| &= 
\limsup_n\max_{z\in\cK} 
 \left| \sum_{k=1}^\infty (z^2 \E W_{11}^2)^k \frac{\tr (S^{(n)})^k}{k} \right|
  \\ 
 &\leq \limsup_n\max_{z\in\cK} 
\sum_{k=1}^\infty |z|^{2k} \frac{\tr (S^{(n)})^k}{k} 
 < \infty, 
\end{align*} 
and the bounds in~\eqref{bnd-kappa} hold true. Regarding $F_n(z)$, we can 
check by, \emph{e.g.}, a moment calculation that for each $n >0$, it holds that
\[
\limsup_{k\to\infty} \frac{|Z_k|^{1/k} (\tr (S^{(n)})^k)^{1/(2k)}}{k^{1/(2k)}} 
  \leq \rho(S^{(n)})^{1/2}  \quad \text{w.p. } 1 
\]
therefore, $F_n(z)$ is well-defined as a $\HHd$--valued random variable for
all large $n$. Moreover, we can see that
$\E|F_n(z)|^2 = \sum_{k\geq 1} |z|^{2k} \tr (S^{(n)})^k / k 
  = - \log\det(1 - |z|^2 S^{(n)})$ which is upper bounded on the compacts of 
$D(0,1)$.  Therefore $(F_n)$ is tight. 

As a consequence, the function $g_n(z) = \kappa_n(z) \exp(-F_n(z))$ is a
well-defined $\HHd$-valued random variable, and the sequence $(g_n)$ is tight.
Write $g_n(z) = 1 + \sum_{k\geq 1} G^{(n)}_k (-z)^k$, and recall that $q_n(z) =
1 + \sum_{k=1}^n P^{(n)}_k (-z)^k$. We need to show that for each fixed integer
$k > 0$, the asymptotic equivalence 
\begin{equation}
\label{P-G} 
(P^{(n)}_1,\ldots, P^{(n)}_k) \sim_n (G^{(n)}_1,\ldots, G^{(n)}_k) 
\end{equation} 
holds true. By applying Propositions~\ref{tight_seq} and~\ref{propcvg}, 
Theorem~\ref{th-chpol} will then be proven. 

Recalling the discussion that precedes Proposition~\ref{fin-dim}, we can 
write $(P^{(n)}_1,\ldots, P^{(n)}_k) = Q(\tr X^{(n)}, \ldots, \tr (X^{(n)})^k)$
for some polynomial $Q$ independent of $n$. Notice also that $g_n$ can be 
written as 
\[
g_n(z) = \exp\left( 
  - \sum_{k=1}^{\infty} z^k \left( Z_k \sqrt{\frac{\tr (S^{(n)})^k}{k}} 
    + \frac{\bm^{(n)}_k}{k} \right) \right). 
\] 
Therefore, by applying the same argument as for $q_n$, we obtain that
\[
(G^{(n)}_1,\ldots, G^{(n)}_k) = Q\Bigl( Z_1 \sqrt{\tr S^{(n)}} + \bm^{(n)}_1, 
  \ldots,  Z_k \sqrt{k \tr (S^{(n)})^k} + \bm^{(n)}_k \Bigr), 
\]
for the same polynomial $Q$, and \eqref{P-G} follows from 
Proposition~\ref{fin-dim}. 

\subsection{Releasing the boundedness assumption on $W_{11}$. End of the proof 
 of Theorem~\ref{th-chpol}} 

To finish the proof of Theorem~\ref{th-chpol}, all what remains to prove is the
truth of the asymptotic equivalence~\eqref{P-G} when $W_{11}$ has a second
moment without any additional assumption. As
in~\cite{bordenave2021convergence}, we truncate the $W_{ij}$'s by writing 
$W_{ij}^{(M)} = \1_{|W_{i,j}|\leq M}W_{i,j} -\E \1_{|W_{i,j}|\leq M}W_{i,j}$
for some $M > 0$, and we show that the truncation error is negligible when $M$
is large. 
One of the ideas of~\cite{bordenave2021convergence} is that it is 
much easier to control the effect of this truncation on the coefficients 
$P^{(n)}_k$ rather than on the traces $\tr (X^{(n)})^k$, as it is frequently 
done in random matrix theory. 
\begin{lem}
\label{P-PM} 
For $M > 0$, let $W_{ij}^{(M)}$ be defined as 
\[
W_{ij}^{(M)} = 
  \1_{|W_{i,j}|\leq M}W_{i,j} -\E \1_{|W_{i,j}|\leq M}W_{i,j}. 
\]
Define the matrix 
  $X^{(n,M)} = \begin{bmatrix} X_{i,j}^{(n,M)} \end{bmatrix}_{i,j\in[n]}$ as  
$X_{i,j}^{(n,M)}:=\sqrt{s_{i,j}^{(n)}}W^{(M)}_{i,j}$ for $i,j\in[n]$. 
For $k\in [n]$, let $P_{k}^{(n,M)}$ be given as 
\[
P_{k}^{(n,M)}= \sum_{I \subset [n]: |I|=k} \det X^{(n,M)}_I 
          \quad \text{for} \ k \in [n]. 
\]
Then, for each fixed integer $k > 0$, the bound 
\[
\sup_n \E|P_{k}^{(n)}-P_{k}^{(n,M)}|^2 \leq \varepsilon_M  
\]
holds true, with $\varepsilon_M \to 0$ as $M\to\infty$. 
\end{lem}
\begin{proof}
Recalling that the polynomial $p^{(n)}(w) = \perm\left( I + w S^{(n)} \right)$
introduced in the proof of Lemma~\ref{perm} can be written as $p^{(n)}(w) = 1 +
\sum_{k=1}^n p^{(n)}_k w^k$ with $$p^{(n)}_k = \sum_{I\subset[n], |I|= k} \perm
S^{(n)}_I$$ for $k\in[n]$, we write 
\begin{align*}
\E|P_{k}^{(n)}-P_{k}^{(n,M)}|^2 &= 
\sum_{I \subset [n]:|I|=k} 
             \E \left| \det X^{(n)}_I - \det X^{(n,M)}_I \right|^2 \\
 &= \sum_{I \subset [n]:|I|=k} \sum_{\sigma\in\mathfrak S_I} 
  \E \left| \prod_{i\in I} X_{i,\sigma(i)}^{(n)} 
  -  \prod_{i\in I} X_{i,\sigma(i)}^{(n,M)} \right|^2 \\ 
 &= \E \left| W_{1,1}\cdots W_{1,k} 
  - W_{1,1}^{(M)} \cdots W_{1,k}^{(M)} \right|^2 
  \sum_{I \subset [n]:|I|=k} \sum_{\sigma\in\mathfrak S_I} 
   \prod_{i\in I} s_{i,\sigma(i)}  \\
&= \E \left| W_{1,1}\cdots W_{1,k} -
    W_{1,1}^{(M)} \cdots W_{1,k}^{(M)} \right|^2 p^{(n)}_k. 
\end{align*} 
Observing that 
\[
\sup_n p^{(n)}_k \leq \frac{\sup_n p^{(n)}(1/2)}{(1/2)^k} < \infty 
\]
by Lemma~\ref{perm}, we can clearly choose $\varepsilon_M$ as 
\[
\varepsilon_{M} = \E \left| W_{1,1}\cdots W_{1,k} -
    W_{1,1}^{(M)} \cdots W_{1,k}^{(M)} \right|^2 \sup_n p^{(n)}_k  . 
\]
\end{proof}

Fix an integer $k > 0$, and let $\varphi : \C^k \to \R$ be a Lipschitz and
bounded function.  Define the sequence of independent Gaussian random variables
$(Z^{(M)}_\ell)_{\ell\geq 1}$ as $\E Z_\ell^{(M)} =0$, $\E |Z_\ell^{(M)} |^2=
\E |W_{11}^{(M)} |^2$, and $\E[(Z_\ell^{(M)})^2] =
\E[(W^{(M)}_{11})^2]^{\ell}$.  By a slight modification of the proof of
Proposition~\ref{fin-dim}, we obtain an asymptotic equivalence similar
to~\eqref{as-eq}, namely, for $M > 0$ large enough, 
\[
\left( \tr X^{(n,M)}, \ldots, \tr (X^{(n,M)})^{k} \right) \sim_n  
\left( \sqrt{\tr S^{(n)}} Z_{1}^{(M)} + \bm_{1}^{(n,M)}, \ldots, 
 \sqrt{k \tr (S^{(n)})^{k}} Z_{k}^{(M)} + \bm_{k}^{(n,M)} \right). 
\] 
where $\bm^{(n,M)}_k$ has the same expression as $\bm_k^{(n)}$ with 
$W_{11}$ being replaced with $W_{11}^{(M)}$. As in the last sub-section, we 
have $(P^{(n,M)}_1,\ldots, P^{(n,M)}_k) = 
  Q(\tr X^{(n,M)}, \ldots, \tr (X^{(n,M)})^k)$, where $Q$ is a polynomial 
independent of $n$. Therefore the quantity 
\[
 \E\varphi(P^{(n,M)}_1,\ldots, P^{(n,M)}_k) 
  -  \E\varphi\circ Q\Bigl(
 \sqrt{\tr S^{(n)}} Z_{1}^{(M)} + \bm_{1}^{(n,M)}, \ldots, 
 \sqrt{k \tr (S^{(n)})^{k}} Z_{k}^{(M)} + \bm_{k}^{(n,M)} \Bigr).
\]
tends to $0$ as $n$ tends to $\infty.$

We also know that 
\[
\E\varphi(G^{(n)}_1,\ldots, G^{(n)}_k) = 
 \E \varphi\circ Q\Bigl( Z_1 \sqrt{\tr S^{(n)}} + \bm^{(n)}_1, 
  \ldots,  Z_k \sqrt{\frac{\tr (S^{(n)})^k}{k}} + \bm^{(n)}_k \Bigr). 
\]
Now, by the previous lemma, it holds that 
\[
\sup_n \left| \E\varphi(P^{(n,M)}_1,\ldots, P^{(n,M)}_k) 
 - \E\varphi(P^{(n)}_1,\ldots, P^{(n)}_k) \right|  
   \xrightarrow[M\to\infty]{} 0 .
\]
Moreover, using that the traces $\tr (S^{(n)})^\ell$ are bounded by numbers 
independent of $n$, we also have 
\begin{multline*} 
\sup_n \left| 
 \E \varphi\circ Q\Bigl( \sqrt{\tr S^{(n)}} Z_1 + \bm^{(n)}_1, 
  \ldots,  \sqrt{k \tr (S^{(n)})^k} Z_k + \bm^{(n)}_k\Bigr) \right. \\ 
 \left.  - \E \varphi\circ Q\Bigl( 
  \sqrt{\tr S^{(n)}} Z_1^{(M)} + \bm^{(n,M)}_1, 
  \ldots,  \sqrt{k \tr (S^{(n)})^k} Z_k^{(M)} + \bm^{(n,M)}_k\Bigr) 
\right| 
  \xrightarrow[M\to\infty]{} 0 .
\end{multline*} 
It results that 
\[
 \E\varphi(P^{(n)}_1,\ldots, P^{(n)}_k) 
 - \E\varphi(G^{(n)}_1,\ldots, G^{(n)}_k)  \xrightarrow[n\to\infty]{} 0 , 
\] 
and the asymptotic equivalence~\eqref{P-G} is established when $W_{11}$ 
satisfies our general assumptions. This concludes the proof of 
Theorem~\ref{th-chpol}.

\subsection{Proof of Theorem~\ref{th-main} from Theorem~\ref{th-chpol}}
\label{subs_proof_of_spectral_radius} 

Denoting as $\overline D(0,r)$ the closed centered disk of $\C$ of radius $r$,
the probability event $[ \rho(X^{(n)}) \geq 1 + \varepsilon ]$ satisfies 
\[
\left[ \rho(X^{(n)}) > 1 + \varepsilon \right] = 
 \left[ \min_{z\in \overline D(0,r)} | q_n(z) | = 0 \right], 
\]
where $r = 1 / (1+\varepsilon)$. Fix $\delta\in(0,1-r)$. The function 
$\HHd \to [0,\infty), f \mapsto \min_{z\in \overline D(0,r)} | f(z) |$ is 
continuous. Therefore, recalling the notation $g_n = \kappa_n \exp(-F_n)$, we 
obtain from Theorem~\ref{th-chpol} that 
\begin{equation}
\label{eq-min} 
\min_{z\in \overline D(0,r)} | q_n(z) | \sim_n 
\min_{z\in \overline D(0,r)} | g_n(z) |. 
\end{equation} 
By Theorem~\ref{th-chpol}, the deterministic sequence $(\kappa_n)$ is 
precompact in $\HHd$ by the normal family theorem, and the random 
$\HHd$--valued sequence $(F_n)$ is tight. Take a sub-sequence, call it $(n)$, 
along which $(g_n)$ converges in distribution towards 
$g_\infty = \kappa_\infty \exp(-F_\infty)$, where $\kappa_\infty \in \HHd$ and 
$F_\infty$ is a $\HHd$--valued random variable. From~\eqref{eq-min}, we obtain 
that $\min_{z\in \overline D(0,r)} | q_n(z) |$ converges in distribution 
to $\min_{z\in \overline D(0,r)} | g_\infty(z) |$ along the sub-sequence $(n)$.
But $g_\infty(z) = 0$ if and only if $\kappa_\infty(z) = 0$, and furthermore, 
the equation $\kappa_\infty(z) = 0$ has no solution on $D(0,1-\delta)$ 
by~\eqref{bnd-kappa}. Therefore,
\[
\lim_n \PP\left[ \min_{z\in \overline D(0,r)} | q_n(z) | > 0 \right] = 1, 
\]
and Theorem~\ref{th-main} is proven.

\section{Proofs for Section~\ref{sec-app}} 
\label{prf-cases} 

\subsection{Proof of Proposition~\ref{TC}} 
Let $\log$ be the branch of the logarithm on the open disk $D(1,1)$ such that
$\log(1) = 0$, and observe that 
$|\log(1+z)| = | \sum_{i=1}^\infty z^i / i | \leq |z| / (1-|z|)$ for $|z| < 1$.
Denoting as $\{ \lambda_1, \lambda_2, \ldots \}$ the (eigenvalue) spectrum of 
the compact operator $\bS$, and recalling that $\bS$ is trace-class with 
$\rho(\bS) \leq 1$, it holds that 
\[
\forall \gamma\in [0,1), \quad 
\sum_{k=1}^\infty \left|  \log ( 1 - \gamma \lambda_k ) \right|  
  \leq  \sum_{k=1}^\infty \frac{\gamma |\lambda_k|}{1- \gamma|\lambda_k|} 
 \leq \frac{\gamma}{1-\gamma} \sum_{k=1}^\infty |\lambda_k| < \infty. 
\] 
Therefore, the Fredholm determinant $\det( \bI - \gamma \bS)$ can be 
written for $\gamma\in[0,1)$ as 
\[
\det( \bI - \gamma \bS) 
= \prod_{k=1}^\infty \left( 1 - \gamma \lambda_k\right), 
\] 
and satisfies $\min_{\gamma\in[0,1-\varepsilon]}\det( \bI - \gamma \bS) > 0$
for $\varepsilon\in(0,1]$. 

We can also express $\det( \bI - z \bS)$ as an analytic expansion for $z\in\C$. 
Namely, consider the series 
\[
 f(z) = 1 + \sum_{k=1}^\infty (-1)^k \bd_k z^k, \quad z \in \C, 
\]
where 
\[ 
\bd_k = \frac{1}{k !} 
 \int_0^1 \int_0^1 \cdots \int_0^1 
 \det \bS\begin{pmatrix} x_1 & x_2 & \cdots x_k \\  
  x_1 & x_2 & \cdots x_k \end{pmatrix} dx_1 dx_2 \ldots dx_k . 
\]
Since, by Hadamard's inequality, 
 $| \bd_k | \leq k^{k/2} \| \bS \|_\infty^k / k!$, the
function $f$ is entire, and it is well-known to coincide with $\det( \bI - z
\bS)$. Now, as in the proof of Lemma~\ref{perm}, we interpret the matrix
$S^{(n)}$ as a piece-wise constant approximation $\bS^{(n)} : [0,1]^2 \to \R_+$
of the operator $\bS$, by writing 
$\bS^{(n)}(x,y) = n s^{(n)}_{ij} = \bS(i/n, j/n)$ when $(x,y) \in \left[
\frac{i-1}{n}, \frac in \right) \times \left[ \frac{j-1}{n}, \frac jn \right),
i,j \in[n]$ (completions on the right and upper borders irrelevant). With this,
we have 
\[
\det\left( I - z S^{(n)}\right) = \det\left(\bI - z \bS^{(n)}\right) = 
  1 + \sum_{k=1}^n (-1)^k \bd_k^{(n)} z^k
\]
where 
\[ 
\bd^{(n)}_k = \sum_{I\subset[n], |I|= k} \det S^{(n)}_I 
 = \frac{1}{k !} 
 \int_0^1 \int_0^1 \cdots \int_0^1 
 \det \bS^{(n)} \begin{pmatrix} x_1 & x_2 & \cdots x_k \\  
  x_1 & x_2 & \cdots x_k \end{pmatrix} dx_1 dx_2 \ldots dx_k . 
\]
Similarly to $\bd_k$, it holds that 
$|\bd^{(n)}_k | \leq \| \bS \|_\infty^k / k !$. Furthermore, 
$\bd^{(n)}_k \to_n \bd_k$ for each $k$ thanks to the continuity of the kernel
$\bS$. With this at hand, one can check that the sequence of polynomials 
$\left( \det\left( I - z S^{(n)}\right) \right)_n$ is bounded on the compacts
of $\C$ and converges point-wise to $\det\left(\bI - z \bS\right)$. 
Thus, this convergence is uniform on the compacts of $\C$, and 
Proposition~\ref{TC} follows. 

\subsection{Proof of Proposition~\ref{ER}}

Write $S^{(n)} = [ S^{(n)}_{ij} ]_{i,j=1}^n$. For $i\in[n]$, define $b_i^{(n)}$ 
as 
\[
b_i^{(n)} = 
\sum_{j=1}^n \E S^{(n)}_{ij} = 
\frac 1n \sum_{j=1}^n \bS(i/n,j/n) \leq \|\bS\|_\infty. 
\]
By Chernoff's theorem \cite[Th.~2.3.1]{vershynin2018high}, it holds that 
\[
\PP\left[ \sum_{j=1}^n S^{(n)}_{ij} \geq t \right] = 
\PP\left[ \sum_{j=1}^n B^{(n)}_{ij} \geq t K_n \right] 
 \leq \left(\frac{e b_i^{(n)}}{t} \right)^{t K_n} 
 \leq \left(\frac{e \|\bS\|_\infty}{t} \right)^{t K_n} 
\] 
for $t > 0$ large enough. By the union bound, we therefore have 
\[
\PP\left[ \vertiii{{S}^{(n)}} \geq t \right] 
 \leq \exp\left(\log n + t K_n \log( e \|\bS\|_\infty / t ) \right) .
\]
Using that $K_n \geq \log n$, choosing $t$ large enough and invoking the 
Borel-Cantelli lemma, we obtain the first assertion of Proposition~\ref{ER}. 

To prove the second assertion, we work on the reverse characteristic 
polynomials  
\[
q_n^S(z) = \det(I_n - z S^{(n)}). 
\]
Following the general canvas of the proof of Theorem~\ref{th-chpol}, we
consider $(q^S_n)$ as a sequence of $\HH$--valued random variables. We first 
establish the tightness of this sequence. Second, 
we show that for each fixed integer $k > 0$, it holds that 
\begin{equation} 
\label{S-barS}  
\tr (S^{(n)})^k - \tr (\uS^{(n)})^k \toprobalong 0
\end{equation} 
By an obvious continuity argument, we further know that 
$\tr (\uS^{(n)})^k \to_n \tr \bS^k$, where $\tr \bS^k$ is given 
by~\eqref{trop}. As a consequence, 
\[
\tr (S^{(n)})^k \toprobalong \tr \bS^k. 
\]
The tightness and these convergences for each $k > 0$ lead to the convergence 
$q_n^S(z) \toprobashort \det(\bI - z \bS)$ in $\HH$ following the approach
developed in the proof of Theorem~\ref{th-chpol}, and the 
convergence~\eqref{q->S} follows. 

To establish the tightness of $(q_n^S)$, we write as in \eqref{q-P}
\[
q_n^S(z)= 1 + \sum_{k=1}^{n} (-z)^k 
  \sum_{I \subset [n], |I| = k} \det S^{(n)}_I . 
\]
By the triangle inequality,
\[
\E |q_n^S(z)| 
\leq 1 + \sum_{k=1}^{n} |z|^k \frac{1}{K_n^k}
\sum_{I \subset [n], |I| = k} 
 \sum_{\sigma \in \mathfrak{S}_I} 
 \E \prod_{i \in I} B^{(n)}_{i, \sigma(i)} = \perm(I_n + |z| \uS^{(n)}).
\]
Remembering that $\min_{\gamma\in[0,1-\varepsilon]} \det (\bI - \gamma\bS) > 0$
and using an argument similar to the proof of Lemma~\ref{perm}, we obtain 
that $\perm(I_n + |z| \uS^{(n)})$ is uniformly bounded on the compacts of
$D(0,1)$. By Proposition~\ref{cond-tight}, we obtain that $(q_n^S)$ is 
tight.  

It remains to establish the convergence~\eqref{S-barS} to finish the proof. 

In what follows, let \( \mathbf{P}_k \) denote the isomorphic classes (in the
classical graph-theoretic sense) of closed walks which use \( k \) edges. We
allow the walks to have multiple edges and loops.  For each \( P \in
\mathbf{P}_k \) and $n \geq k$, we shall denote \( P([n]) \) the set of walks
belonging to the isomorphic class of \( P \) and having vertices in \( [n] \).
Given a matrix $M = [M_{ij}]$, recall the notation 
$M_{\bJ} = M_{i_1i_2} \ldots M_{i_k i_1}$ for $\bJ = (i_1,\ldots,i_k)$. We
start by writing  
\[
\E \tr(S^{(n)})^k = \frac{1}{K_n^k} 
\sum_{P \in \mathbf{P}_k} \sum_{\bJ \in P([n])} \E B^{(n)}_{\bJ}. 
\]
Now, for any walk \( P \in \mathbf{P}_k \), we have:
\[
\frac{1}{K_n^k} \sum_{\bJ \in P([n])} \E B^{(n)}_{\bJ} 
 \leq \frac{C}{K_n^k} n^{| \{ \text{vertices of } P \}|}  
 \left( \frac{K_n}{n} \right)^{| \{ \text{distinct edges of } P\}| }, 
\]
where the \( |\{ \text{distinct edges of } P\}| \) counts each edge of 
\( P \) once, ignoring repetitions.

Furthermore, the number of vertices of \( P \) is bounded by the number of
distinct edges of \( P \). This is true because we can start a closed walk in
\( P \) and assign each distinct edge to a vertex, starting from the first edge
during the walk that starts from that vertex.
Since the number of distinct edges of \( P \) is clearly bounded by \( k \), we
conclude that for a walk \( P \) to have a non-negligible contribution in the
above sum, the number of distinct edges should equal the number of vertices.
Thus, the non-negligible paths \( P \) should satisfy \( k = \text{number of
vertices of } P = \text{number of distinct edges of } P \).
Thus, denoting as $\cD_k$ the sub-set of $[n]^k$ defined as 
\[
\cD_k = \{ (i_1, \dots, i_k) \in [n]^k : 
   \forall l \neq j \in [k], i_l \neq i_j \}, 
\]
we obtain that 
\[
\E \tr(S^{(n)})^k = \sum_{\bJ \in \cD_k} \uS^{(n)}_{\bJ} + o_n(1). 
\]
Now, we know that the matrix $\uS^{(n)}$ complies with 
Assumptions~\ref{bnd-S} and~\ref{bnd-det} by replacing the $K_n$ in the 
statement of Assumption~\ref{bnd-S} with $n$. Therefore, by Lemma~\ref{errS}, 
it holds that 
\[
0 \leq \tr(\uS^{(n)})^k - \sum_{\bJ \in \cD_k} \uS^{(n)}_{\bJ} \leq C / n.
\]
Thus:
\[
\E \tr(S^{(n)})^k = \tr(\uS^{(n)})^k + o_n(1). 
\]
To obtain~\eqref{S-barS}, it remains to show that 
$\var \tr(S^{(n)})^k \to_n 0$. We have here 
\[
\var \tr(S^{(n)})^k = \frac{1}{K_n^{2k}} 
\sum_{P_1,P_2 \in \mathbf{P}_k} \sum_{\bJ_1 \in P_1([n]),\bJ_2 \in P_2([n])}
\E[(B^{(n)}_{\bJ_1} - \E B^{(n)}_{\bJ_1})
    (B^{(n)}_{\bJ_2} - \E B^{(n)}_{\bJ_2})]. 
\]
Clearly, the summand is zero when the walks $P_1$ and $P_2$ have no common
vertex. When these walks have a common vertex, we have by a similar argument as
above that 
\begin{align*}
& \frac{1}{K_n^{2k}} 
\sum_{\bJ_1 \in P_1([n]),\bJ_2 \in P_2([n])}
\E[(B^{(n)}_{\bJ_1} - \E B^{(n)}_{\bJ_1})
    (B^{(n)}_{\bJ_2} - \E B^{(n)}_{\bJ_2})] \\ 
&\leq 
\frac{C}{K_n^{2k}}  \left(
\frac{K_n}{n}\right)^{|\{\text{distinct edges of } P_1\}| + 
 |\{\text{distinct edges of } P_2\}| }
 n^{|\{\text{vertices of } P_1 \}| + 
 |\{ \text{vertices of } P_2\}| -1} \\
&\leq \frac{C}{n} . 
\end{align*}
Thus, $\var \tr(S^{(n)})^k \to_n 0$, and Proposition~\ref{ER} is proven. 

\subsection{Corollary~\ref{ERnovap}: Sketch of proof}  
\label{prf-cor} 
Given a small $\varepsilon > 0$, let $\beta > 0$ be such that 
$\det\left( \bI - (1 - \varepsilon/2) \bS \right) \geq \beta$, and define the 
probability event 
\[
\cE_n = \left[ \vertiii{S^{(n)}} \leq 2 C_S \ \text{and} \ 
 \min_{\gamma \in [0, 1 - \varepsilon / 2]} 
  \det\left( I_n - \gamma S^{(n)} \right) \geq \beta / 2 \right]. 
\]
Then we have 
\[
\PP\left[ \rho(X^{(n)}) \geq 1 + \varepsilon \right] \leq 
\PP\left[ [ \rho(X^{(n)}) \geq 1 + \varepsilon ] \ | \ \cE_n \right] 
 + \PP \left[ \cE_n^{\text{c}} \right]. 
\]
By adjusting $\delta$ in the statement and the proof of Theorem~\ref{th-chpol},
we obtain by a slight modification of the proof of Theorem~\ref{th-main} that 
$\PP\left[ [ \rho(X^{(n)}) \geq 1 + \varepsilon ] \ | \ \cE_n \right] \to_n 0$. 
It is clear by Proposition~\ref{ER} that 
$\PP \left[ \cE_n^{\text{c}} \right] \to_n 0$, and we obtain the 
spectral confinement result 
$\PP\left[ \rho(X^{(n)}) \geq 1 + \varepsilon \right] \to_n 0$. 

\subsection{Proposition~\ref{Reg}: Sketch of proof} 
As for the previous proposition, we show that 
$q_n^S(z) = \det(I_n - z S^{(n)})$ converges in probability to 
$\det(\bI - z \bS)$ in $\HH$. 

Using that the rows of $R^{(n)}$ are independent and that 
$\E R^{(n)}_{ij} = K_n / n$, thus, 
$\E S^{(n)}_{ij} = \uS^{(n)}_{ij}$, we obtain that 
$\E|q_n^S(z)| \leq \perm(I + |z| \uS^{(n)})$, hence the tightness of 
$(q_n^S)$.  

We now quickly show that $\tr (S^{(n)})^k \toprobashort \tr\bS^k$ for 
each fixed integer $k > 0$. 
Let $E^{(n)} = [ E^{(n)}_{ij} ]_{i,j=1}^n$ be a random matrix with 
i.i.d.~Bernoulli elements such that $\E E^{(n)}_{11} = K_n / n$, and define 
the random matrix $Y^{(n)} = [ Y^{(n)}_{ij} ]_{i,j=1}^n$ as 
$Y^{(n)}_{ij} = (n/K_n) \uS^{(n)}_{ij} E^{(n)}_{ij}$. Along the principle of 
the proof of Proposition~\ref{ER}, one can prove that 
$\E \tr (Y^{(n)})^k \to_n \tr\bS^k$ and $\var\tr (Y^{(n)})^k \to_n 0$. To show 
that $\tr (S^{(n)})^k \toprobashort \tr\bS^k$, it is enough to show that 
$\E \tr (S^{(n)})^k - \E \tr (Y^{(n)})^k \to_n 0$ and 
$\E(\tr (S^{(n)})^k)^2 - \E(\tr (Y^{(n)})^k)^2 \to_n 0$. 
 
Given a fixed integer $\ell > 0$ and integers $i_1,\ldots i_\ell \in [n]$ 
which are all different, it holds that 
\[
\E R^{(n)}_{1i_1} \ldots R^{(n)}_{1i_\ell} = 
 \PP[ (R^{(n)}_{1i_1}, \ldots, R^{(n)}_{1i_\ell}) = (1,\ldots,1)] 
 = \frac{\binom{n-\ell}{K_n-\ell}}{\binom{n}{K_n}} 
 = \frac{(K_n-\ell+1)\times\cdots\times K_n}{(n-\ell+1)\times\cdots\times n},   
\]
while 
\[
\E E^{(n)}_{1i_1} \ldots E^{(n)}_{1i_\ell} = \left(\frac{K_n}{n}\right)^\ell.
\]
We therefore have after a small calculation that 
\begin{equation} 
\label{R-B} 
0 \leq 
 \E E^{(n)}_{1i_1} \ldots E^{(n)}_{1i_\ell} -  
 \E R^{(n)}_{1i_1} \ldots R^{(n)}_{1i_\ell} \leq 
 \frac{C_\ell}{K_n} \E E^{(n)}_{1i_1} \ldots E^{(n)}_{1i_\ell} , 
\end{equation} 
where $C_\ell$ depends on $\ell$ only. Re-using the notations of the proof of 
Proposition~\ref{ER}, we now have 
\begin{align*} 
0\leq \E \tr(Y^{(n)})^k - \E \tr (S^{(n)})^k  &= \frac{n^k}{K_n^k} 
\sum_{P \in \mathbf{P}_k} \sum_{\bJ \in P([n])} 
\uS^{(n)}_{\bJ}\left( \E E^{(n)}_{\bJ} - \E R^{(n)}_{\bJ} \right) \\ 
 &\leq 
\frac{C}{K_n} \frac{n^k}{K_n^k} 
   \sum_{P \in \mathbf{P}_k} \sum_{\bJ \in P([n])} 
   \uS^{(n)}_{\bJ} \E E^{(n)}_{\bJ} \\
 &= \frac{C}{K_n} \E \tr(Y^{(n)})^k , 
\end{align*}  
where, to obtain the inequality, we decomposed $\E R^{(n)}_{\bJ}$ into a 
product of expectations over the different rows of $R^{(n)}$, and we used
the bound~\eqref{R-B}. 
Since $\E \tr (Y^{(n)})^k \to_n \tr\bS^k$, we obtain that 
$\E\tr(S^{(n)})^k - \E\tr(Y^{(n)})^k\to_n 0$. 
The proof that 
$\E(\tr (S^{(n)})^k)^2 - \E(\tr (Y^{(n)})^k)^2 \to_n 0$ is similar. 

\nocite{*} 


\begin{funding}
M.L.~received funding by the European Union’s Horizon 
 2020 research and innovation program under the
 Marie Sklodowska-Curie grant agreement No 101034255.
\end{funding}
\end{document}